\documentclass[11pt,reqno]{amsart}
\usepackage{}
\usepackage{mathrsfs}
\usepackage{amsfonts}
\usepackage{a4wide}
\allowdisplaybreaks \numberwithin{equation}{section}
\usepackage{color}

\newtheorem{theorem}{Theorem}[section]
\newtheorem{proposition}[theorem]{Proposition}

\newtheorem{lemma}[theorem]{Lemma}

\theoremstyle{definition}

\newcommand{\R}{\mathbb{R}}
\newcommand{\ds}{\displaystyle}

\begin{document}
\title[New type solutions]
{Positive and sign-changing solutions for the nonlinear
Schr\"{o}dinger systems with synchronization and separation
 }
 \author{ Qingfang Wang and Wenju Wu}

\address{School of Mathematics and Computer Science, Wuhan Polytechnic University, Wuhan 430079, P. R. China }
\email{wangqingfang@whpu.edu.cn}

\address{School of Mathematics and Statistics, Central China Normal University, Wuhan, 430079, P. R. China}
\email{wjwu@mails.ccnu.edu.cn}
\thanks{The research was supported by NSFC (No.~12126356).}
\date{\today}
\begin{abstract}

In this paper, we consider the following nonlinear Schr\"odinger system:
\begin{eqnarray*}\begin{cases}
-\Delta u+P(x)u=\mu_1u^3+\beta uv^2, \,\,\,\,\,x\in\R^3,\cr
-\Delta v+Q(x)v=\mu_2v^3+\beta u^2v, \,\,\,\,\,x\in\R^3,
\end{cases}
\end{eqnarray*}
where $P(x),Q(x)$ are positive radial potentials,~$\mu_1,\,\mu_2>0$,~$\beta\in\R$ is a coupling constant.
We constructed a new type of solutions which are different from the ones obtained in \cite{PW}. This new family of solutions to system
have a more complex concentration structure and are centered at the points lying on the top and the bottom circles of a cylinder with height $h$. Moreover, we examine the effect of nonlinear coupling on the solution structure.
In the repulsive case, we construct an unbounded sequence of non-radial positive vector solutions of segregated type.  In the attractive case, we construct an unbounded sequence of non-radial positive vector solutions of synchronized type. Moreover, we prove that there exist infinitely many sign-changing solutions whose energy can be arbitrarily large.

\end{abstract}
\maketitle
\section{Introduction}
We consider the following nonlinear Schr\"odinger system
\begin{eqnarray}\begin{cases}\label{eqs1.1}
-\Delta u+P(x)u=\mu_1u^3+\beta uv^2, \,\,\,\,\,x\in\R^3,\cr
-\Delta v+Q(x)v=\mu_2v^3+\beta u^2v, \,\,\,\,\,x\in\R^3,
\end{cases}
\end{eqnarray}
where $P(x)$ and $Q(x)$ are continuous positive radial functions, $\mu_1>0,\,\mu_2>0$ and $\beta\in \R$ is a coupling constant.
These types of systems arise when one consider standing wave solutions of time-dependent $N$-coupling  Schr\"odinger systems with $N=2$ of the form
\begin{eqnarray}\begin{cases}
-i\frac{\partial}{\partial t}\Phi_j=\Delta\Phi_j-V_j\Phi_j+\mu_j|\Phi_j|^2\Phi_j+\Phi_j\sum\limits_{l\neq j}\beta_{jl}|\Phi_l|^2, \,\,\,\,\,&x\in\R^3,\cr
\Phi_j=\Phi_j(x,t)\in C, \,\,\,\,\,t>0,\,&j=1,2,\cdots,N,
\end{cases}
\end{eqnarray}
where $\mu_j$ and $\beta_{jl}=\beta_{lj}$ are constants. These system of equations, also known as Gross-Pistaevskii equations, have applications in many physical problems such as in nonlinear optics and in Bose-Einstein condensates theory for multispecies Bose-Einstein condensates.
For $N=2,$ it arises in the Hartree-Fock theory for a double condensates, that is a binary mixture of a Bose-Einstein condensate in two different hyperfine states \cite{EGBB,E}.
Physically, $\Phi_1$ and $\Phi_2$ are the wave functions of the corresponding condensates, $\mu_1$,\,$\mu_2$ and $\beta$ are the intraspecies and interspecies scattering lengths, respectively. The sign of the scattering length $\beta$ determines whether the interactions of states are repulsive or attractive. In the attractive case the components of a vector solution tend to go along with each other, leading to synchronization. In the repulsive case, the components tend to segregated from each other, leading to phase separations. These phenomena have been documented in experiments as well as in numeric simulations, for example \cite{CLLL,MS,CH} and references therein.

For the understanding of \eqref{eqs1.1}, let us begin with the single equation. It's well known \cite{Kwong} that for any $1<p<2^*=\frac{2N}{N-2}$, the equation
\begin{eqnarray}\begin{cases}\label{eqs1.2}
 -\Delta w + w= w^p ,\ w>0  \mbox{ in }\;\R^N,\cr
 w(0)=\max\limits_{\R^d}w(x),  \ w \in H^1(\R^N)
 \end{cases}
\end{eqnarray}
has a unique solution denoted by $W$ which is radial, and there exists $C_{N, p} >0$ such that
\begin{align}\label{wy1}
W^*(|x|) = C_{N, p}(1+O(|x|^{-1}))|x|^{\frac{1-N}{2}}e^{-|x|} \quad\hbox{as}~ |x|\rightarrow\infty.
\end{align}
If we replace the term $w$ in \eqref{eqs1.2} by $Vw$ with a nonnegative potential $V$, the situation is drastically different. In the pioneer work \cite{WY}, Wei and Yan constructed infinitely many nonradial positive solutions of the nonlinear Schr\"odinger equation
\begin{align}\label{eqs1.03}
 -\Delta w + V(x)w=w^{p} \ \mbox{ in }\;\R^N, \quad w\in H^1(\R^N).
 \end{align}
Here $1<p<2^*$, $V$ is nonnegative, continuous and radial satisfying
$$V(|x|)=1+\frac{a}{|x|^m}+O\Big(\frac{1}{|x|^{m+\sigma}}\Big), \quad \mbox{as }\; |x|\rightarrow+\infty,$$
with $a>0$, $m>1$ and $\sigma>0$. In fact, the solutions constructed in \cite{WY} happen to be non-degenerate, in the sense that the linearized operators around these solutions are invertible. This property makes it possible to use the solutions as new blocks to generate new constructions \cite{GMPY1}.
Under the same assumptions, Duan and Musso in \cite{DM} constructed different type solutions of the equation \eqref{eqs1.03}.

\medskip
Later on, the two components system has been studied extensively in the literature. Let $\mu_1, \mu_2 > 0$, consider the following system
%first the following system (see \cite{BDW}):
\begin{eqnarray}\begin{cases}\label{eqs1.3}
 -\Delta u + u=\mu_1 u^3+\beta v^2u,\,\,\,x\in \R^3,\cr
 -\Delta v + v=\mu_2 v^3+\beta u^2v,\,\,\,x\in \R^3.
 \end{cases}
 \end{eqnarray}
It's easy to see that there exist radial solutions of \eqref{eqs1.3} as follows:
\begin{align}
\label{uv}
 (U,V)=(\alpha W,\gamma W)
\end{align}
provided $-\sqrt{\mu_1\mu_2}<\beta<\min\{\mu_1,\mu_2\}$ or $\beta>\max\{\mu_1,\mu_2\}$, and
\begin{align}
\label{ag}
 \alpha=\sqrt{\frac{\mu_2-\beta}{\mu_1\mu_2-\beta^2}} ,\ \ \gamma=\sqrt{\frac{\mu_1-\beta}{\mu_1\mu_2-\beta^2}}.
\end{align}
If radial potentials $P_j$ are involved, Peng and Wang proved that under suitable conditions on $P_j$, the Schr\"odinger system
\begin{eqnarray}\begin{cases}\label{eqs1.3bis}
 -\Delta u + P_1u=\mu_1 u^3+\beta v^2u,\,\,\,x\in \R^3,\cr
 -\Delta v + P_2v=\mu_2 v^3+\beta u^2v,\,\,\,x\in \R^3
 \end{cases}
 \end{eqnarray}
has infinitely many non-radial positive solutions of segregated type or synchronized type in $\R^3$. A key ingredient of their study is the nondegeneracy of solutions $(U, V)$ given by \eqref{uv}. More precisely, they showed that there exists a sequence $(\beta_i) \subset(-\sqrt{\mu_1\mu_2},0)$ satisfying $\lim_{i\to\infty}\beta_i= -\sqrt{\mu_1\mu_2}$ such that for any
\begin{align}\label{lambda}
\beta \in \Lambda := \Big[(-\sqrt{\mu_1\mu_2},0)\backslash\{\beta_i\}\Big] \cup(0,\min\{\mu_1,\mu_2\})\cup(\max\{\mu_1,\mu_2\},\infty),
\end{align}
$(U,V)$ is non-degenerate for the system \eqref{eqs1.3}, in the sense that the kernel in $(H^1(\R^3))^2$ of the linearized system to \eqref{eqs1.3} at $(U, V)$ is given by
$${\rm Span}\Big\{\Big(\frac{\alpha}{\gamma}\frac{\partial W}{\partial x_j},\frac{\partial W}{\partial x_j}\Big),\; j=1,2,3 \Big\}$$ with $(\alpha, \gamma)$ in \eqref{ag}.

About $P=Q=1$, for the attractive case that is $\beta>0$, it is known that there are special positive solutions with two components being positive constant multiples of the unique positive solution of the scalar cubic Schr\"odinger equation for example \cite{BW1,BW2}.
The work of \cite{LW} gives solutions with one component peaking at the origin and the other having a finite number of peaks on a $k-$ polygon. In the symmetric case, \cite{WW} gives infinitely many non-radial positive solutions for $\beta\leq-1$. For the linearly coupled nonlinear Schr\"odinger systems. Lin and Peng \cite{LP} consider the segregated vector solutions.
Later on, many works are realized for the two or more components of systems. We refer the readers to \cite{AC,ACB1,BW1,BDW,DWW,FL,NTTV,PWW,PV,WZ,WY,WW,WW2,WY1} and the references therein.

We assume that $P(r)>0$,\,$Q(r)>0$ and $\beta$ satisfying the following conditions
\begin{itemize}
\item[$(P)$] There are constants $a\in\R$, $m>1$ and $\theta_1>0$ such that as $r\rightarrow+\infty$,
$$
P(r)=1+\frac{a}{r^m}+O(\frac{1}{r^{m+\theta_1}}).
$$
\item[$(Q)$] There are constants $b\in\R$, $n>1$ and $\theta_2>0$ such that as $r\rightarrow+\infty$,
$$
Q(r)=1+\frac{b}{r^n}+O(\frac{1}{r^{n+\theta_2}})
$$
\item[$(\beta_2)$]There exists $\beta^*>0$ such that $\beta<\beta^*$.
\end{itemize}

Our main results in this paper can be stated as follows:
\begin{theorem}\label{th1.1}
Suppose that $P(r)$ satisfies $(P)$ and $Q(r)$ satisfies $(Q)$. For any positive integer $k$, $\beta$ satisfies \eqref{lambda}, problem \eqref{eqs1.1} has infinitely many non-radial positive synchronized solutions $(u_k,v_k)$ whose energy can be arbitrarily large provided one of the following conditions holds:
\begin{itemize}
\item[ (i)] $m<n$, $a>0$ and $b\in\R$; or $m>n$, $a\in \R$ and $b>0$.
\item[(ii)] $m=n$, $aA_1+bA_2>0$ where $A_1$ and $A_2$ are defined in Proposition \ref{proA.4}.
\end{itemize}
\end{theorem}

\begin{theorem}\label{th1.2}
Suppose that $P(r)$ satisfies $(P)$, $Q(r)$ satisfies $(Q)$ and $m=n$, $a>0$, $b>0$, $\beta$ satisfies $(\beta_2)$. Problem \eqref{eqs1.1} has infinitely many non-radial positive segregated solutions $(u_k,v_k)$ whose energy can be arbitrarily large.
\end{theorem}
Our methods also allow us to treat sign-changing solutions. The solutions in Theorems 1.1 and 1.2 are constructed in the form with $u,v$ components both having alternating sign-changing bumps at infinity. For sign-changing solutions, we have the following results.
\begin{theorem}\label{th1.31}
Suppose that $P(r)$ satisfies $(P)$ and $Q(r)$ satisfies $(Q)$. Then there exists a decreasing sequence $\beta_k\subset(-\sqrt{\mu_1\mu_2},0)$ satisfying $\lim\limits_{k\to\infty}\beta_k= -\sqrt{\mu_1\mu_2}$ such that for any
\begin{align}\label{eqs110}
\beta \in \Lambda := \Big[(-\sqrt{\mu_1\mu_2},0)\backslash\{\beta_k\}\Big] \cup(0,\min\{\mu_1,\mu_2\})\cup(\max\{\mu_1,\mu_2\},\infty),
\end{align}
problem \eqref{eqs1.1} has infinitely many non-radial sign-changing synchronized solutions, provided the condition in Theorem 1.1 holds.
\end{theorem}
\begin{theorem}\label{th1.41}
Suppose that $P(r)$ satisfies $(P)$, $Q(r)$ satisfies $(Q)$ and $m=n$, $a>0$, $b>0$, $\beta$ satisfies $(\beta_2)$. Problem \eqref{eqs1.1} has infinitely many non-radial sign-changing segregated solutions $(u_k,v_k)$ whose energy can be arbitrarily large. Furthermore, $\lim\limits_{k+\infty}\max u_k>0,\, \lim\limits_{k+\infty}\max v_k>0$.
\end{theorem}
Hereafter, for any function $K(x)>0$, the Sobolev space $H_K^1(\R^3)$ is endowed with the standard norm
$$
\|u\|_K=\Bigl(\int_{\R^3}|\nabla u|^2+K(x)u^2dx\Bigr)^{\frac{1}{2}},
$$
which is induced by the inner product
$$
\langle u,v\rangle_K=\int_{\R^3}\left(\nabla u\nabla v+K(x)uv\right).
$$
Define $H$ to be the product space $H_P^1(\R^3)\times H_Q^1(\R^3)$ with the norm
$$
\|(u,v)\|=\|u\|_P+\|v\|_Q.
$$
Let
\begin{align}\label{add1}
\bar{x}_j=r\Bigl(\sqrt{1-h^2}\cos\frac{2(j-1)\pi}{k},\sqrt{1-h^2}\sin\frac{2(j-1)\pi}{k},h\Bigr),j=1,2,\cdots,k,
\end{align}
\begin{align}\label{add2}
\underline{x}_j=r\Bigl(\sqrt{1-h^2}\cos\frac{2(j-1)\pi}{k},\sqrt{1-h^2}\sin\frac{2(j-1)\pi}{k},-h\Bigr),j=1,2,\cdots,k.
\end{align}
The parameters $h$ and $r$ are positive numbers and are chosen in the range
\begin{align*}
h\in[\alpha_0\frac{1}{k},\alpha_1\frac{1}{k}],\,\,\,\,\,\,\,\,r\in (C_1k\ln k,C_2k\ln k),
\end{align*}
for $\alpha_0,\alpha_1,C_1,C_2$ are positive constants, independent of $k$. We define the approximate solutions as
$$
U_{r,h}=\sum\limits_{j=1}^kU_{\bar{x}_j}+\sum\limits_{j=1}^kU_{\underline{x}_j},\,\,\,\,
V_{r,h}=\sum\limits_{j=1}^kV_{\bar{x}_j}+\sum\limits_{j=1}^kV_{\underline{x}_j},
$$
where $U_{\bar{x}_j}=U(y-\bar{x}_j)$, $U_{\underline{x}_j}=U(y-\underline{x}_j)$.

We will verify Theorem \ref{th1.1} by proving the following result.
\begin{theorem}\label{th1.3}
Under the assumptions of Theorem \ref{th1.1}, there is an integer $k_0>0$ such that for any integer $k\geq k_0$, \eqref{eqs1.1} has a solution $(u_k,v_k)$ of the form
$$
(u_k,v_k)=(U_{r,h}+\varphi_k, V_{r,h}+\psi_k),
$$
where $(\varphi_k,\psi_k)\in H_P\times H_Q$, $r\in (C_1k\ln k,C_2k\ln k)$ as $k\rightarrow+\infty$.
\end{theorem}
Let $W_{\mu}$ be the unique solution of the following problem
\begin{eqnarray}\begin{cases}
-\Delta w+w=\mu w^3,w>0,\hbox{in}~\R^3,\cr
w(0)=\max\limits_{x\in\R^3}w(x),w(x)\in H^1(\R^3).
\end{cases}
\end{eqnarray}
It is well known that $W_\mu$ is non-degenerate and $W_\mu(x)=W_{\mu}(|x|)$, $W_\mu'<0$. We will use $(W_{\mu_1}, W_{\mu_2})$ to build up the approximate solutions for \eqref{eqs1.1}.

Let
$$
W_{\mu_1,h}=\sum\limits_{j=1}^kW_{\mu_1,\bar{x}_j}+\sum\limits_{j=1}^kW_{\mu_1,\underline{x}_j},\, W_{\mu_2,h}=\sum\limits_{j=1}^kW_{\mu_2,\bar{y}_j}+\sum\limits_{j=1}^kW_{\mu_2,\underline{y}_j},
$$
where
\begin{align}\label{add3}
\bar{y}_j=\rho\Bigl(\sqrt{1-h^2}\cos\frac{(2j-1)\pi}{k},
\sqrt{1-h^2}\sin\frac{(2j-1)\pi}{k},h\Bigr),
\end{align}
\begin{align}\label{add4}
\underline{y}_j=\rho\Bigl(\sqrt{1-h^2}
\cos\frac{(2j-1)\pi}{k},\sqrt{1-h^2}\sin\frac{(2j-1)\pi}{k},-h\Bigr).
\end{align}
To prove Theorem \ref{th1.2}, we need to prove the following result.
\begin{theorem}\label{th1.4}
Under the assumptions of Theorem \ref{th1.2}, there exists an integer $k_0>0$ such that for any integer $k\geq k_0$, system \eqref{eqs1.1} has a solution $(u_k,v_k)$ of the form
$$
(u_k,v_k)=(W_{r,h}+\bar{\varphi}_k,W_{\rho,h}+\bar{\psi}_k),
$$
where $(\bar{\varphi}_k,\bar{\psi}_k)\in H_P\times H_Q$, $r\in (C_1k\ln k,C_2k\ln k),\,\rho\in(C_3k\ln k,C_4 k\ln k)$ as $k\rightarrow+\infty,$ $C_1,C_2,C_3,C_4$ are positive constants.
\end{theorem}
We will prove Theorems \ref{th1.3} and \ref{th1.4}
by using the finite dimensional reduction method.
Although this method is very standard, we have to make some modifications.
Also, there involve some technical and complicated estimates caused by
the nonlinear coupled terms. We would like to point out that
the main distinction between our construction and the ones in \cite{PW} is that there are two parameters $r,h$ to choose in the locations $\bar{x}_j,\underline{x}_j$ of the bumps in \eqref{add1} and \eqref{add2}.
Also, there are two parameters $\rho,h$ to choose in the locations $\bar{y}_j,\underline{y}_j$ of the bumps in \eqref{add3} and \eqref{add4}. So we need to overcome the difficulty caused by the new parameter $h.$ In fact, expansions for the main term of energy functional are determined by the order of $h$(see Propositions \ref{proA.4} and \ref{proA.5}).

Define
\begin{align*}
H'=\Bigl\{u:u\in H^1(\R^N),u\,\hbox{is even in}\,& x_j,j=2,3,\,\cr
&u(r\cos\theta,r\sin\theta,x'')=(-1)^iu(r\cos(\theta+\frac{\pi i}{k}),r\sin(\theta+\frac{\pi i}{k}),x'')\Bigr\}.
\end{align*}
Denote
\begin{align*}
(\bar{U}_{r,h},\bar{V}_{r,h})=\Bigl((-1)^{i-1}U_{r,h},(-1)^{i-1}V_{r,h}\Bigr).
\end{align*}
For sign-changing solutions to problem \eqref{eqs1.1}, we have the following result.
\begin{theorem}
Under the assumptions of Theorem 1.3, system \eqref{eqs1.1} has infinitely many non-radial synchronized sign-changing solutions
\begin{align*}
(u,v)=(\bar{U}_{r,h}+\varphi^*_k,\bar{V}_{r,h}+\psi^*_k)=\Bigl((-1)^{i-1}U_{r,h}+\varphi_k^*,(-1)^{i-1}V_{r,h}+\psi_k^*\Bigr).
\end{align*}
\end{theorem}
\begin{theorem}
Under the assumptions of Theorem 1.4, system \eqref{eqs1.1} has infinitely many non-radial segregated sign-changing solutions
\begin{align*}
(u,v)=(\bar{W}_{r,h}+\tilde{\varphi}_k,\bar{W}_{\rho,h}+\tilde{\psi}_k)=\Bigl((-1)^{i-1}W_{r,h}+\tilde{\varphi}_k,(-1)^{i-1}W_{\rho,h}+\tilde{\psi}_k\Bigr).
\end{align*}
\end{theorem}
This paper is organized as follows. In section 2, we will carry out the reduction to a finite dimensional setting and prove Theorem 1.5. The study of the existence of segregated solutions for system \eqref{eqs1.1} and the proof of Theorem 1.6 will appear in Section 3. We conclude with the energy expansion in the appendix.

\section{synchronized vector solutions and the proof of Theorem \ref{th1.1}}
In this section, we consider synchronized vector solutions and prove Theorem \ref{th1.1} by proving Theorem \ref{th1.3}.
For $j=1,2,\cdots,k$, we divide $\R^3$ into $k$ parts
\begin{align*}
\Omega_j=\Bigl\{y=(y_1,y_2,y_3)\in\R^3:\langle\frac{(y_1,y_2)}{|(y_1,y_2)|},
(\cos\frac{2(j-1)\pi}{k},\sin\frac{2(j-1)\pi}{k})\rangle\geq\cos\frac{\pi}{k}\Bigr\},
\end{align*}
where $\langle\cdot,\cdot\rangle$ denotes the scalar product in $\R^2$. For $\Omega_j$, we divide it into two kinds of points
\begin{align*}
\Omega_j^+=\{y:y=(y_1,y_2,y_3)\in\Omega_j,y_3\geq0\};\,\,\,\,
\Omega_j^-=\{y:y=(y_1,y_2,y_3)\in\Omega_j,y_3\leq0\}.
\end{align*}
We see that $\R^3=\bigcup\limits_{j=1}^k\Omega_j$, $\Omega_j=\Omega_j^+\bigcup\Omega_j^-$, and the interior of $\Omega_i\bigcap\Omega_j$, $\Omega_j^+\bigcap\Omega_j^-$ are empty sets for $i\neq j$.

We now define the symmetric Sobolev space
\begin{align*}
H_s=&\Bigl\{u:u\in H^1(\R^N), u\, \hbox{is even in}\,y_2, \cr &u(\sqrt{y_1^2+y_2^2}\cos\theta,\sqrt{y_1^2+y_2^2}\sin\theta,y_3)=u(\sqrt{y_1^2+y_2^2}\cos(\theta+\frac{2\pi j}{k}),\sqrt{y_1^2+y_2^2}\sin(\theta+\frac{2\pi j}{k}),y_3)\Bigr\},
\end{align*}
where $\theta=\arctan\frac{y_2}{y_1}$ is a polar coordinate in $\R^2$.

In this paper, we always assume
\begin{align*}
(r,h)\in&\Bigl(\bigl(\frac{m}{2\pi}-\delta\bigr)k\ln k,\,\bigl(\frac{m}{2\pi}+\delta\bigr)k\ln k\Bigr)\times\Bigl(\bigl(\frac{\pi(m+2)}{m}-\delta_1\bigr)\frac{1}{k},\,\bigl(\frac{\pi(m+2)}{m}+\delta_1\bigr)\frac{1}{k}\Bigr)\cr
=:&S_k\times D_h,
\end{align*}
for some $\delta,\delta_1>0$ small and independent of $k$. Let
\begin{align*}
\bar{Y}_{1j}=\frac{\partial U_{\bar{x}_j}}{\partial r},\, \underline{Y}_{1j}=\frac{\partial U_{\underline{x}_j}}{\partial r},\,\bar{Y}_{2j}=\frac{\partial U_{\bar{x}_j}}{\partial h},\,\underline{Y}_{2j}=\frac{\partial U_{\underline{x}_j}}{\partial h};
\end{align*}
\begin{align*}
\bar{Z}_{1j}=\frac{\partial V_{\bar{x}_j}}{\partial r}, \,\underline{Z}_{1j}=\frac{\partial V_{\underline{x}_j}}{\partial r},\,\bar{Z}_{2j}=\frac{\partial V_{\bar{x}_j}}{\partial h},\,\underline{Z}_{2j}=\frac{\partial V_{\underline{x}_j}}{\partial h}.
\end{align*}

Define
\begin{align*}
I(u,v)=\frac{1}{2}\int_{\R^3}(|\nabla u|^2+P(|x|)u^2+|\nabla v|^2+Q(|x|)v^2)-\frac{1}{4}\int_{\R^3}(\mu_1|u|^4+\mu_2|v|^4)-\frac{\beta}{2}\int_{\R^3}u^2v^2.
\end{align*}
Then $I\in C^2(\R,\R)$ and its critical points are solutions of \eqref{eqs1.1}.
We define the constrained space
\begin{align*}
E=\Bigl\{(u,v),(u,v)\in H_s\times H_s,&\sum\limits_{j=1}^k\int_{\R^N}W_{\bar{x}_j}^2(\bar{Y}_{lj}u+\bar{Z}_{lj}v)=0,\cr
\,&\,\,\,\,\,\,\,\,\,\,\,\,\,\,\,\,\,\,\,\,\,\,\,\,\sum\limits_{j=1}^k\int_{\R^N}W_{\underline{x}_j}^2(\b{Y}_{lj}u+\b{Z}_{lj}v)=0,\,\,l=1,2\Bigr\}.
\end{align*}

Let
\begin{align*}
J(\varphi,\psi)=I(U_{r,h}+\varphi,V_{r,h}+\psi),(\varphi,\psi)\in E.
\end{align*}
Expand $J(\varphi,\psi)$ as follows:
\begin{equation}\label{eqs2.1}
J(\varphi,\psi)=J(0,0)+\ell(\varphi,\psi)+\frac{1}{2}L(\varphi,\psi)+R(\varphi,\psi),\,\,\,(\varphi,\psi)\in E,
\end{equation}
where
\begin{align*}
\ell(\varphi,\psi)=&\int_{\R^3}\Bigl[(P(|x|)-1)U_{r,h}\varphi+(Q(|x|)-1)V_{r,h}\psi\Bigr]+\mu_1\int_{\R^3}\Bigl[U_{r,h}^3-\bigl(\sum\limits_{j=1}^kU_{\bar{x}_j}^3
+\sum\limits_{j=1}^kU_{\underline{x}_j}^3\bigr)\Bigr]\varphi\cr
&+\mu_2\int_{\R^3}\Bigl[V_{r,h}^3-(\sum\limits_{j=1}^kV_{\bar{x}_j}^3+\sum\limits_{j=1}^kV_{\underline{x}_j}^3)\Bigr]\psi
-\beta\int_{\R^3}\Bigl[U_{r,h}^2V_{r,h}\varphi-(\sum\limits_{j=1}^kV_{\bar{x}_j}U_{\bar{x}_j}^2+V_{\underline{x}_j}U_{\underline{x}_j}^2)\varphi\Bigr]\cr
&-\beta\int_{\R^3}\Bigl[U_{r,h}V_{r,h}^2\psi-(\sum\limits_{j=1}^kV_{\bar{x}_j}^2U_{\bar{x}_j}+V_{\underline{x}_j}^2U_{\underline{x}_j})\psi\Bigr],
\end{align*}
\begin{align*}
L(\varphi,\psi)=\int_{\R^3}\Bigl(|\nabla\varphi|^2+P(|x|)\varphi^2+|\nabla \psi|^2+\Bigr.&\Bigl.Q(|x|)\psi^2\Bigr)-\int_{\R^3}\Bigl(3\mu_1U_{r,h}^2\varphi^2+3\mu_2V_{r,h}^2\psi^2\Bigr)\cr
&-\beta\int_{\R^3}\Bigl(U_{r,h}^2\psi^2+4U_{r,h}V_{r,h}\varphi\psi+V_{r,h}^2\varphi^2\Bigr),
\end{align*}
and
\begin{align*}
R(\varphi,\psi)=&\int_{\R^3}\left[\mu_1\left(U_{r,h}+\varphi\right)^4+\mu_2\left(V_{r,h}+\psi\right)^4-\mu_1U_{r,h}^4-\mu_2V_{r,h}^4\right.\cr
&\left.-4\mu_1U_{r,h}^3\varphi-4\mu_2V_{r,h}^3\psi-6\mu_1U_{r,h}^2\varphi^2-6\mu_2V_{r,h}^2\psi^2\right]\cr
&-\frac{\beta}{2}\int_{\R^3}\left[\left(U_{r,h}+\varphi\right)^2\left(V_{r,h}+\psi\right)^2-U_{r,h}^2V_{r,h}^2
-2\left(U_{r,h}^2V_{r,h}\psi+U_{r,h}V_{r,h}^2\varphi\right)\right.\cr
&\left.-2\left(U_{r,h}^2\psi^2+V_{r,h}^2\varphi^2+4U_{r,h}V_{r,h}\varphi\psi\right)\right].
\end{align*}

It is easy to check that
\begin{align*}
&\int_{\R^3}\left(\nabla u\nabla\varphi+P(|x|)u\varphi-3\mu_1U_{r,h}^2u\varphi\right)+\int_{\R^3}\left(\nabla v\nabla\psi+Q(|x|)v\psi-3\mu_2V_{r,h}^2v\psi\right)\cr
&-\beta\int_{\R^3}\left(U_{r,h}^2v\psi+V_{r,h}u\varphi+2U_{r,h}V_{r,h}u\psi+2U_{r,h}V_{r,h}v\psi\right)
\end{align*}
is a bounded bi-linear functional in $E$. Thus, there is a bounded linear operator $L$ from $E$ to $E$, such that
\begin{align*}
\langle L(u,v),(\varphi,\psi)\rangle=&\int_{\R^3}\left(\nabla u\nabla\varphi+P(|x|)u\varphi-3U_{r,h}^2u\varphi\right)+\int_{\R^3}\left(\nabla v\nabla\psi+Q(|x|)v\psi-3V_{r,h}^2v\psi\right)\cr
&-\beta\int_{\R^3}\left(U_{r,h}^2v\psi+V_{r,h}u\varphi+2U_{r,h}V_{r,h}u\psi+2U_{r,h}V_{r,h}v\psi\right).
\end{align*}
From the above analysis, we have the following two lemmas.
\begin{lemma}\label{lm2.1}
There is a constant $C>0$, independent of $k$, such that
\begin{align*}
\|L(u,v)\|\leq C\|(u,v)\|, \,\,\,\,\,\,\,\,\forall\,(u,v)\in E.
\end{align*}
\end{lemma}
Next, we discuss the invertibility of $L$.
\begin{lemma}\label{lm2.2}
There is a constant $\varrho>0$, independent of $k$ such that for any $r\in D_k$
\begin{align*}
\|L(u,v)\|\geq \varrho\|(u,v)\|, \,\,\,\,\,\,\,\,\forall\,(u,v)\in E.
\end{align*}
\end{lemma}
\begin{proof}
We argue by contradiction. Suppose that when $k\rightarrow+\infty$, there exist $h_k,r_k\in S_k$, $(u_k,v_k)\in E$ satisfying
\begin{align*}
\langle L(u_k,v_k),(\varphi,\psi)\rangle=o_k(1)\|(u_k,v_k)\|(\varphi,\psi)\|,\,\,\,(\varphi,\psi)\in E.
\end{align*}
We may assume that $\|(u_k,v_k)\|^2=k$. For convenience, we use $r$ to denote $r_k$. By symmetry, we see that
\begin{align}\label{eqs2.5}
\langle L(u_k,v_k),(\varphi,\psi)\rangle=&k\int_{\Omega_1}\left(\nabla u\nabla\varphi+P(|x|)u\varphi-3U_{r,h}^2u\varphi\right)+\int_{\Omega_1}\left(\nabla v\nabla\psi+Q(|x|)v\psi-3V_{r,h}^2v\psi\right)\cr
&-\beta\int_{\R^3}\left(U_{r,h}^2v\psi+V_{r,h}u\varphi+2U_{r,h}V_{r,h}u\psi+2U_{r,h}V_{r,h}v\psi\right)\cr
=&O_k(1)(\sqrt{k})\|(\varphi,\psi)\|,
\end{align}
and
\begin{align*}
\int_{\Omega_1}\left(|\nabla u_k|^2+P(|x|)u_k^2+|\nabla v_k|^2+Q(x)v_k^2\right)=1.
\end{align*}
Letting
$$\bar{u}_k(x)=u_k(x-\bar{x}_1),\,\,\,\bar{v}_k=v_k(x-\bar{x}_1),$$
for any $R>0$, since $|\bar{x}_2-\bar{x}_1|=2r\sqrt{1-h^2}\sin\frac{\pi}{k}\geq \frac{m}{4}\ln k$, we can choose $k$ large enough such that $B_R(\bar{x}_1)\subset\Omega_1$. As a result, we have
\begin{align*}
\int_{B_R(0)}\left(|\nabla \bar{u}_k|^2+P(|x|)\bar{u}_k^2+|\nabla \bar{v}_k|^2+Q(|x|)\bar{v}_k^2\right)\leq C.
\end{align*}
So we may assume that there exist $u,v\in H^1(\R^3)$ such that as $k\rightarrow+\infty$,
\begin{align*}
&\bar{u}_k\rightarrow u,\,\,\hbox{weekly in}\, H_{loc}^1(\R^3),\,\,\,\,\,\,\,\,\,\,\,\,\,\,\,\,\, \bar{u}_k\rightarrow u,\hbox{strongly in}\, L_{loc}^2(\R^3).\cr
&\bar{v}_k\rightarrow v,\,\,\hbox{weekly in}\, H_{loc}^1(\R^3),\,\,\,\,\,\,\,\,\,\,\,\,\,\,\,\,\, \bar{v}_k\rightarrow v,\hbox{strongly in}\, L_{loc}^2(\R^3).
\end{align*}
From the orthogonal conditions for functions in $E$,
\begin{align*}
\int_{\R^3}U_{\bar{x}_1}^2\Bigl(\frac{\partial U_{\bar{x}_1} }{\partial r}u_k+\frac{\partial V_{\bar{x}_1} }{\partial r}v_k\Bigr)=0,
\end{align*}
and the identity
\begin{align*}
\frac{\partial U_{\bar{x}_1}}{\partial r}=\sqrt{1-h^2}\frac{\partial U_{\bar{x}_1}}{\partial y_1}+h\frac{\partial U_{\bar{x}_1}}{\partial y_3},
\end{align*}
we can get
\begin{align}\label{eqs2.7}
&\sqrt{1-h^2}\int_{\R^3}W_{\bar{x}_1}^2\frac{\partial U_{\bar{x}_1}}{\partial y_1}u_k+h\int_{\R^3}W_{\bar{x}_1}^2\frac{\partial U_{\bar{x}_1}}{\partial y_3}u_k\cr
&\,\,\,\,\,\,\,\,\,\,\,\,\,\,\,\,\,\,\,\,\,\,\,\,\,\,\,\,\,\,\,\,\,\,\,\,\,\,\,\,\,\,\,\,\,\,\,\,+\sqrt{1-h^2}\int_{\R^3}W_{\bar{x}_1}^2\frac{\partial V_{\bar{x}_1}}{\partial y_1}v_k+h\int_{\R^3}W_{\bar{x}_1}^2\frac{\partial V_{\bar{x}_1}}{\partial y_3}v_k=0.
\end{align}
Similarly, combining
\begin{align*}
\int_{\R^3}W_{\bar{x}_1}^2\Bigl(\frac{\partial U_{\bar{x}_1}}{\partial h}u_k+\frac{\partial V_{\bar{x}_1}}{\partial h}v_k\Bigr)=0
\end{align*}
and
\begin{align*}
\frac{\partial U_{{\bar{x}_1}}}{\partial h}=-\frac{hr}{\sqrt{1-h^2}}\frac{\partial U_{{\bar{x}_1}}}{\partial y_1}+r\frac{\partial U_{{\bar{x}_1}}}{\partial y_3},
\end{align*}
we can get
\begin{align}\label{eqs2.8}
&\frac{h}{\sqrt{1-h^2}}\int_{\R^3}W_{\bar{x}_1}^2\frac{\partial U_{\bar{x}_1}}{\partial y_1}u_k
-\int_{\R^3}W_{\bar{x}_1}^2\frac{\partial U_{\bar{x}_1}}{\partial y_3}u_k
+\frac{h}{\sqrt{1-h^2}}\int_{\R^3}W_{\bar{x}_1}^2\frac{\partial V_{\bar{x}_1}}{\partial y_1}v_k\cr
&\,\,\,\,\,\,\,\,\,\,\,\,\,\,\,-\int_{\R^3}W_{\bar{x}_1}^2\frac{\partial V_{\bar{x}_1}}{\partial y_3}v_k=0.
\end{align}
From \eqref{eqs2.7} and \eqref{eqs2.8}, we have
\begin{align}\label{eqs2.07}
\int_{\R^3}W^2\Bigl(\frac{\partial U}{\partial y_1}u+\frac{\partial V}{\partial y_1}v\Bigr)=0,\,\,\,\int_{\R^3}W^2\Bigl(\frac{\partial U}{\partial y_3}u+\frac{\partial V}{\partial y_3}v\Bigr)=0.
\end{align}
Now, we claim that $(u,v)$ satisfies
\begin{eqnarray}\begin{cases}\label{eqs2.40}
-\Delta u+u-3\mu_1 U^2u-\beta V^2u-2\beta UVv=0,\,\,\,x\in\R^3,\cr
-\Delta v+v -3\mu_2 V^2v-\beta U^2v-2\beta UVu=0,\,\,\,x\in\R^3.
\end{cases}
\end{eqnarray}
Define
\begin{align*}
\tilde{E}^+=\Bigl\{(\varphi,\psi)\in H^1(\R^3)\times H^1(\R^3),\int_{\R^3}W^2(\frac{\partial U}{\partial y_1}u+\frac{\partial V}{\partial y_1}v)=0,\int_{\R^3}W^2(\frac{\partial U}{\partial y_3}u+\frac{\partial V}{\partial y_3}v)=0\Bigr\}.
\end{align*}

For any $R>0$, let $(\varphi,\psi)\in C_0^\infty(B_R(0))\times C_0^\infty(B_R(0))\bigcap\tilde{E}^+$ and be even in $x_2$. Then denote
\begin{align*}
(\varphi_k,\psi_k)=(\varphi(y-\bar{x}_1),\psi(y-\bar{x}_1))\in C_0^\infty(B_R(\bar{x}_1))\times C_0^\infty(B_R(\bar{x}_1))\subset\Omega_1,
\end{align*}
if $k$ is large enough. We may identify $(\varphi_k,\psi_k)$ as elements in $E$ by the values outside $\Omega_1$ with symmetry. With the argument in\cite{PW}, we find
\begin{align}\label{eqs2.10}
\int_{\Omega_1}\left(\nabla u_k\nabla\varphi_k+P(x)u_k\varphi_k-3\mu_1 U_{r,h}^2u_k\varphi_k\right)
\rightarrow\int_{\R^3}\left(\nabla u\nabla\varphi+u\varphi-3\mu_1 U^2u\varphi\right),
\end{align}
\begin{align}\label{eqs2.11}
\int_{\Omega_1}\left(\nabla v_k\nabla\psi_k+Q(x)v_k\psi_k-3\mu_2 V_{r,h}^2v_k\psi_k\right)
\rightarrow\int_{\R^3}\left(\nabla v\nabla\psi+v\psi-3\mu_2 V^2v\psi\right),
\end{align}
and
\begin{align}\label{eqs2.12}
&\int_{\Omega_1}\left(U_{r,h}^2v_k\psi_k+V_{r,h}^2u_k\varphi_k+2U_{r,h}V_{r,h}u_k\psi_k+2U_{r,h}V_{r,h}v_k\psi_k\right)\cr
&\,\,\,\,\,\,\,\,\,\,\,\,\,\,\,\,\,\,\,\,\,\,\,\,\,\,\,\,\,\,\,\,\,\,\,\,\,\,\,\,\,\,\,\,\,\,\,\,\rightarrow
\int_{\R^3}\left(U^2v\psi+V^2u\varphi+2UVu\psi+2UVv\varphi\right).
\end{align}
Inserting \eqref{eqs2.10}-\eqref{eqs2.12} into \eqref{eqs2.5}, we see
\begin{align}\label{eqs2.13}
&\int_{\R^3}\left(\nabla u\nabla\varphi+P(|x|)u\varphi-3U^2u\varphi\right)+\int_{\R^3}\left(\nabla v\nabla\psi+Q(|x|)v\psi-3V^2v\psi\right)\cr
&-\beta\int_{\R^3}\left(U^2v\psi+Vu\varphi+2UVu\psi+2UVv\psi\right)=0.
\end{align}

However, since $u$ and $v$ are even in $x_2$, \eqref{eqs2.13} holds for any function $(\varphi,\psi)\in C_0^\infty(B_R(0))\times C_0^\infty(B_R(0))\bigcap \tilde{E}^{+}$. By the density of $C_0^\infty(B_R(0))\times C_0^\infty(B_R(0))$ in $H^1(\R^3)\times H^1(\R^3)$, we see
\begin{align*}
&\int_{\R^3}\left(\nabla u\nabla\varphi+P(|x|)u\varphi-3U^2u\varphi\right)+\int_{\R^3}\left(\nabla v\nabla\psi+Q(|x|)v\psi-3V^2v\psi\right)\cr
&-\beta\int_{\R^3}\left(U^2v\psi+Vu\varphi+2UVu\psi+2UVv\psi\right)=0,\,\forall\,(\varphi,\psi)\in \tilde{E}^{+}.
\end{align*}
Noting that $(U,V)=(\alpha W,\gamma W)$ and $W$ solves \eqref{eqs1.3}, we can verify that \eqref{eqs2.13} holds for $(\varphi,\psi)=(\frac{\partial U}{\partial y_1},\frac{\partial V}{\partial y_1})$ and $(\varphi,\psi)=(\frac{\partial U}{\partial y_3},\frac{\partial V}{\partial y_3})$. Then \eqref{eqs2.13} is true for any $(\varphi,\psi)\in H^1(\R^3)\times H^1(\R^3)$. So we prove \eqref{eqs2.40}.

By using the non-degenerate results for $(U,V)$, we work in the space of functions which are even in $x_2$. Thus
\begin{align}\label{eqs2.9}
(u,v)=\Bigl(c_{11}\frac{\partial U}{\partial y_1}+c_{12}\frac{\partial U}{\partial y_3},c_{21}\frac{\partial V}{\partial y_1}+c_{22}\frac{\partial V}{\partial y_3}\Bigr),
\end{align}
which combining \eqref{eqs2.07} and \eqref{eqs2.9} yields that $c_{11}=c_{12}=c_{21}=c_{22}=0$.

As a result, we have
$$
\int_{B_R(\bar{x}_1)}\left(u_k^2+v_k^2\right)=o(1),\,\,\,\,\,\, \forall\, R>0.
$$
From Lemmas \ref{lmA.1}-\ref{lmA.2}, we have
$$
U_{r,h}\leq Ce^{-\frac{|y-\bar{x}_1|}{2}},\,\,\,\,\,\,\,\, V_{r,h}\leq Ce^{-\frac{|y-\bar{x}_1|}{2}},\,\,\,\,\,\,\,x\in\Omega_1.
$$
Thus
\begin{align}\label{eqs221}
&\int_{\Omega_1}
\left(|\nabla u_k|^2+P(|x|)u_k^2-3\mu_1U_{r,h}^2u_k^2\right)+\int_{\Omega_1}\left(|\nabla v_k|^2+Q(|x|)v_k^2-3\mu_2V_{r,h}^2v_k^2\right)\cr
=&\int_{\Omega_1}\left(|\nabla u_k|^2+P(|x|)u_k^2\right)
+\int_{\Omega_1}\left(|\nabla v_k|^2+Q(|x|)v_k^2\right)+O(e^{-\frac{R}{2}})\int_{\Omega_1}\left(u_k^2+v_k^2\right),
\end{align}
since there holds
\begin{align}\label{eqs222}
\int_{\Omega_1}\left(V_{r,h}^2u_k^2+4U_{r,h}V_{r,h}u_kv_k+U_{r,h}^2v_k^2\right)=o\Bigl(e^{-\frac{R}{2}}\int_{\Omega_1}(u_k^2+v_k^2)\Bigr).
\end{align}
Inserting \eqref{eqs221} and \eqref{eqs222} into \eqref{eqs2.5}, we find
\begin{align*}
o_k(1)=&\int_{\Omega_1}\left(|\nabla u_k|^2+P(|x|)u_k^2-3\mu_1 U_{r,h}^2u_k^2\right)+\int_{\Omega_1}\left(|\nabla v_k|^2+Q(x)v_k^2-3\mu_2V_{r,h}^2v_k^2\right)\cr
&-\beta\int_{\Omega_1}\left(V_{r,h}^2u_k^2+4U_{r,h}V_{r,h}u_kv_k+U_{r,h}^2v_k^2\right)
=1+O\left(e^{-\frac{R}{2}}\right),
\end{align*}
which is impossible for large $k$ and large $R$. As a result, we complete the proof.
\end{proof}
\begin{lemma}\label{lm2.3}
There exists a constant $C>0$, independent of $k$, such that
\begin{align}
\|R^{(i)}(\varphi,\psi)\|\leq C\|(\varphi,\psi)\|^{(3-i)},\,i=0,1,2.
\end{align}
\end{lemma}
\begin{proof}
Calculating directly, we have that for any $(u,v)\in E$, $(\xi,\eta)\in E$,
\begin{align*}
|R(\varphi,\psi)|\leq&\left|\int_{\R^3}\mu_1 U_{r,h}\varphi^3+\mu_2V_{r,h}\psi^3+\frac{\mu_1}{4}\varphi^4+\frac{\mu_2}{4}\psi^4\right|\cr
&+\frac{|\beta|}{2}\int_{\R^3}\left|(U_{r,h}+\varphi)^2(V_{r,h}+\psi)^2-U_{r,h}^2V_{r,h}^2-2(U_{r,h}V_{r,h}^2\varphi+U_{r,h}^2V_{r,h}\psi)\right.\cr
&\left.-2(U_{r,h}^2\psi^2+V_{r,h}^2\varphi^2+4U_{r,h}V_{r,h}\varphi\psi)\right|\cr
\leq&C\int_{\R^3}\left(|\varphi|^3+|\psi|^3+|\varphi|^4+|\psi|^4+\varphi^2|\psi|+\psi^2|\varphi|\right)\cr
\leq&C\|(\varphi,\psi)\|^3+\|(\varphi,\psi)\|^4,
\end{align*}
\begin{align*}
|\langle R'(\varphi,\psi),(u,v)\rangle|=&\left|\int_{\R^3}\left(3\mu_1 U_{r,h}\varphi^2u+\mu_1\varphi^3u-\beta U_{r,h}\psi^2u-2\beta V_{r,h}\varphi^2\psi u-\beta\varphi\psi^2u\right)\right.\cr
&\left.+\left(3\mu_2 V_{r,h}\psi^2v+\mu_2\psi^3v-\beta V_{r,h}\varphi^2v-2\beta U_{r,h}\varphi\psi^2 v-\beta\varphi^2\psi v\right)\right|\cr
\leq& C(\|\varphi\|_P^2+\|\varphi\|_P^3+\|\psi\|_P^2+\|\psi\|_P^3+\|\varphi\|_P^2\|\psi\|_Q+\|\varphi\|_P\|\psi\|_Q^2)\times(\|u\|_P+\|v\|_Q)\cr
\leq &C(\|(\varphi,\psi)\|^2+\|(\varphi,\psi)\|^3)\|(u,v)\|,
\end{align*}
and similarly,
\begin{align*}
|\langle R''(\varphi,\psi)(u,v),(\xi,\eta)\rangle|\leq C(\|(\varphi,\psi)\|+\|(\varphi,\psi)\|^2)\|(u,v)\|\|(\xi,\eta)\|.
\end{align*}
Hence, the results follow.
\end{proof}

\begin{lemma}\label{lm2.4}
There is a constant $C>0$, independent of $k$, such that
\begin{align}\label{eqs211}
\|\ell_k\|\leq C\Bigl(\frac{k}{r^m}+\frac{k}{r^n}+k^{\frac{1}{2}}e^{-\frac{2\pi r}{k}}\frac{k}{r}\Bigr).
\end{align}
\end{lemma}
\begin{proof}
First, we assume that $l$ is even, for $5\leq l\leq\frac{k}{2}$, $y\in \Omega_k^+$. We claim that $|y-\bar{x}_1|\geq (4-\tau)\frac{\pi r}{k}$ for some $\tau>0$ small and provided that $k$ is large enough.

In fact, we have
\begin{align*}
|y-\bar{x}_1|\geq&|\bar{x}_1-\bar{x}_l|-|y-\bar{x}_l|\geq 2r\sqrt{1-h^2}\sin\frac{(l-1)\pi}{k}-\frac{4\pi r}{k}\cr
\geq&(4-\tau)\frac{\pi r}{k},\,\,\,\,\hbox{if}\,\,\, |y-\bar{x}_k|\leq\frac{4\pi r}{k}
\end{align*}
and
\begin{align*}
|y-\bar{x}_1|\geq|y-\bar{x}_l|\geq4\frac{\pi r}{k},\,\,\,\,\,\hbox{if}\,\,\,|y-\bar{x}_l|\geq4\frac{\pi r}{k}.
\end{align*}
Using the symmetric property for the functions, we know that
\begin{align}\label{eqs2.1700}
&\sum\limits_{i=1}^k\int_{\R^N}(P(x)-1)(U_{\bar{x}_i}+U_{\underline{x}_i})\varphi=2k\int_{\R^N}(P(|x|)-1)U_{\bar{x}_1}\varphi\cr
\leq&C k\int_{\bigcup\limits_{l=1}^4(\Omega_l^++\Omega_l^-)}+\int_{\bigcup\limits_{4<l\leq \frac{k}{2}}(\Omega_l^++\Omega_l^-)}(P(|x|)-1)U_{\bar{x}_1}\varphi\cr
\leq& Ck\int_{\bigcup\limits_{l=1}^4\Omega_l^+\bigcap B_{\delta|\bar{x}_1|}(\bar{x}_1)}+\int_{\bigcup\limits_{l=1}^4\Omega_l^+\setminus B_{\delta|\bar{x}_1|}(\bar{x}_1)}
+Ck\int_{\bigcup\limits_{4<l\leq\frac{k}{2}}\Omega_l^+}(P(x)-1)U_{\bar{x}_1}\varphi,
\end{align}
where $\delta_0>0$ is small constant. We can show that there exist positive constants
\begin{align}\label{eqs2.130}
(P(|x|)-1)U_{\bar{x}_1}\leq\frac{c}{|y-\bar{x}_1+\bar{x}_1|^m}e^{-|y-\bar{x}_1|}\leq \frac{C}{r^m}e^{-|y-\bar{x}_1|},\,\,\,\,y\in \bigcup\limits_{l=1}^4\Omega_l^+\bigcap B_{\delta|\bar{x}_1|}(\bar{x}_1),
\end{align}
\begin{align}\label{eqs2.140}
(P(|x|)-1)U_{\bar{x}_1}\leq Ce^{-|y-\bar{x}_1|}\leq Ce^{-\delta r},\,\,\,\,y\in \bigcup\limits_{l=1}^4\Omega_l^+\setminus B_{\delta|\bar{x}_1|}(\bar{x}_1)
\end{align}
and
\begin{align}\label{eqs2.150}
(P(|x|)-1)U_{\bar{x}_1}\leq Ce^{-|y-\bar{x}_1|}\leq Ce^{-(4-\delta)\frac{\pi r}{k}},\,\,\,\,y\in \bigcup\limits_{4<l<\frac{k}{2}}\Omega_l^+.
\end{align}
Combining \eqref{eqs2.1700}-\eqref{eqs2.150}, we get
\begin{align}\label{eqs2.160}
\sum\limits_{i=1}^k\int_{\R^3}(P(x)-1)(U_{\bar{x}_i}+U_{\underline{x}_i})\varphi\leq C\Big(\frac{1}{r^m}+\frac{1}{r^{2m-1-\tau}}\Big)\|\varphi\|_{L^2(\R^3)}.
\end{align}
Similarly, we have
\begin{align}\label{2.170}
\sum\limits_{i=1}^k\int_{\R^3}(Q(x)-1)(V_{\bar{x}_i}+V_{\underline{x}_i})\psi\leq C\Big(\frac{1}{r^n}+\frac{1}{r^{2n-1-\tau}}\Big)\|\psi\|_{L^2(\R^3)}.
\end{align}

Note that it is obvious that
\begin{align}\label{eqs2.180}
U_{\bar{x}_1}+U_{\underline{x}_i}\geq U_{\bar{x}_i}+U_{\underline{x}_i},\,\,\,\,\,\,\,\,y\in\Omega_1,
\end{align}
\begin{align*}
U_{\bar{x}_j}\geq U_{\underline{x}_j},\,\,\,\,\,\,\,\hbox{for}\,\,\,\,\,\,\,\,y\in \Omega_1^+.
\end{align*}
By symmetry, we have
\begin{align}\label{eqs2.20}
&\int_{\R^3}\Big[U_{r,h}^3-\big(\sum\limits_{i=1}^kU_{\bar{x}_i}^3
+\sum\limits_{i=1}^kU_{\underline{x}_i}^3\big)\Big]\varphi
=2k\int_{\Omega_1^+}\Big[U_{r,h}^3
-\big(\sum\limits_{i=1}^kU_{\bar{x}_i}^3
+\sum\limits_{i=1}^kU_{\underline{x}_i}^3\big)\Big]\varphi\cr
\leq & k\int_{\Omega_1^+}\Big[U_{\bar{x}_1}^2\big(\sum\limits_{i=2}^kU_{\bar{x}_i}
+U_{\underline{x}_1}\big)+\big(\sum\limits_{i=2}^kU_{\bar{x}_i}
+U_{\underline{x}_1}\big)^3\Big]|\varphi|\cr
\leq&k^{\frac{1}{2}}\sum\limits_{i=2}^ke^{-|\bar{x}_1-\bar{x}_i|}\|\varphi\|_{L^2(\R^3)}\leq Ck^{\frac{1}{2}}e^{-\frac{2\pi r}{k}}\frac{k}{r}\|\varphi\|_{L^2(\R^3)}.
\end{align}
Finally, noting that $U=\frac{\alpha}{\gamma}V$, we have
\begin{align}\label{eqs2.21}
\int_{\R^3}\Big|U_{r,h}^2V_{r,h}-
\big(\sum\limits_{j=1}^kV_{\bar{x}_j}U_{\bar{x}_j}^2
+V_{\underline{x}_j}U_{\underline{x}_j}^2\big)\Big|\psi
\leq& C\int_{\R^3}\Big|\big[U_{r,h}^3-\big(\sum\limits_{j=1}^kU_{\bar{x}_j}^3
+\sum\limits_{j=1}^kU_{\underline{x}_j}^3\big)\big]\psi\Big|\cr
\leq& Ck^{\frac{1}{2}}e^{-\frac{2\pi r}{k}}\frac{k}{r}\|\psi\|.
\end{align}
The proof of \eqref{eqs211} follows from \eqref{eqs2.160} and \eqref{eqs2.21}.
\end{proof}
\begin{proposition}
There is an integer $k_0>0$ such that for each $k\geq k_0$, there is a $C^1$ map from $D_k$ to $H_{P}\times H_{Q}$:$(\varphi,\psi)=(\varphi(r),\psi(r))$, $r=|x^1|$ satisfying $(\varphi,\psi)\in E$ and
$$\bigr\langle\frac{\partial J(\varphi,\psi)}{\partial(\varphi,\psi)},(g,h)\bigr\rangle=0,\,\,\,\forall \,(g,h)\in E.$$
Moreover, there is a positive constant $C$ such that
\begin{align*}
\|(\varphi,\psi)\|\leq C\Big(\frac{k}{r^m}+\frac{k}{r^n}+k^{\frac{1}{2}}e^{-\frac{2\pi r}{k}}\frac{k}{r}\Big).
\end{align*}
\end{proposition}
\begin{proof}
It follows from Lemma \ref{lm2.1} that $\ell(\varphi,\psi)$ is a bounded linear functional in $E$. Thus, there is an $\ell_k\in E$ such that
\begin{align*}
\ell(\varphi,\psi)=\langle \ell_k,(\varphi,\psi)\rangle.
\end{align*}
Thus, finding a critical point for $J(\varphi,\psi)$ is equivalent to solving
\begin{align}\label{eqs2.45}
\ell_k+L(\varphi,\psi)+R'(\varphi,\psi)=0.
\end{align}
By Lemma 2.2 , $L$ is invertible. Thus \eqref{eqs2.45} can be rewritten as
\begin{align*}
(\varphi,\psi)=A(\varphi,\psi)=-L^{-1}\ell_k-L^{-1}R'(\varphi,\psi).
\end{align*}
Set
\begin{align*}
D=\Bigl\{(\varphi,\psi):(\varphi,\psi)\in E,\|(\varphi,\psi)\|\leq \frac{k^{1+\sigma}}{r^m}+\frac{k^{1+\sigma}}{r^n}+k^{\frac{1+\sigma}{2}}e^{-\frac{2\pi r}{k}}\frac{k}{r}\Bigr\},
\end{align*}
where $\sigma>0$ is small.

From Lemma \ref{lm2.2} to Lemma \ref{lm2.4} for $k$ large, we have
\begin{align*}
\|A(\varphi,\psi)\|\leq& C\Big(\|\ell_k\|+\||(\varphi,\psi)\|^2\Big)
\leq C\Big(\frac{k}{r^m}+\frac{k}{r^n}+k^{\frac{1}{2}}e^{-\frac{2\pi r}{k}}\frac{k}{r}\Big)+\Big(\frac{k^{1+\sigma}}{r^m}+\frac{k^{1+\sigma}}{r^n}+k^{\frac{1+\sigma}{2}}e^{-\frac{2\pi r}{k}}\frac{k}{r}\Big)^2\cr
\leq&\frac{k^{1+\sigma}}{r^m}+\frac{k^{1+\sigma}}{r^n}+k^{\frac{1+\sigma}{2}}e^{-\frac{2\pi r}{k}}\frac{k}{r}
\end{align*}
and
\begin{align*}
\|A(\varphi_1,\psi_1)-A(\varphi_2,\psi_2)\|=&\|L^{-1}R'(\varphi_1,\psi_1)-L^{-1}R'(\varphi_2,\psi_2)\|\cr
\leq& C(\|(\varphi_1,\psi_1)\|+\|(\varphi_1,\psi_1)\|^2)(\|(\varphi_1,\psi_1)-(\varphi_2,\psi_2)\|)\cr
\leq &\frac{1}{2}(\|(\varphi_1,\psi_1)-(\varphi_2,\psi_2)\|).
\end{align*}
Therefore, A maps $D$ into $D$ and is a contraction map. So, by the contraction mapping theorem, there exists $(\varphi,\psi)\in E$ such that
$(\varphi,\psi)=A(\varphi,\psi)$. Finally, there exists $(\varphi,\psi)$ satisfies
\begin{align*}
\|(\varphi,\psi)\|\leq C\Big(\frac{k}{r^m}+\frac{k}{r^n}+k^{\frac{1}{2}}e^{-\frac{2\pi r}{k}}\frac{k}{r}\Big).
\end{align*}
 \end{proof}
Now we are ready to prove Theorem \ref{th1.3}. Let $(\varphi_k,\psi_k)=(\varphi(r),\psi(r))$ be the map obtained in Proposition 2.5. Define
$$
F(r)=I(U_{r,h}+\varphi_k,V_{r,h}+\psi_k),\,\,\,\,\forall\,r\in D_r.
$$
With the same argument used  in \cite{CNY}, we can easily check that for $k$ sufficiently large, if $r$ is a critical point of $F(r)$, then $(U_{r,h}+\varphi_k,V_{r,h}+\psi_k)$ is a critical point of $I$.

Now we are in a position to prove Theorem 1.3.

{\bf Proof of Theorem 1.3} It follows from Lemma \ref{lm2.1} and Lemma \ref{lm2.3} that
$$
\|L(\varphi_k,\psi_k)\|\leq C\|(\varphi_k,\psi_k)\|,\,\,\,\,\,|R(\varphi_k,\psi_k)|\leq C\|(\varphi_k,\psi_k)\|^3.
$$
So, Proposition \ref{proA.4} gives
\begin{align*}
F(r)=&I\left(U_{r,h},V_{r,h}\right)+\ell\left(\varphi_k,\psi_k\right)+\frac{1}{2}\bigr\langle L\left(\varphi_k,\psi_k\right),\left(\varphi_k,\psi_k\right)\bigr\rangle+R\left(\varphi_k,\psi_k\right)\cr
=&I(U_{r,h},V_{r,h})+O(\|\ell_k\|\|(\varphi_k,\psi_k)\|+\|(\varphi_k,\psi_k)\|^2)\cr
=&I(U_{r,h},V_{r,h})+O_k(\frac{1}{k^{m+1+\sigma}})\cr
=&k\Bigl(A_0+\frac{A_1}{r^m}+\frac{A_2}{r^n}-2C_\beta \frac{k}{r}e^{-2\pi\sqrt{1-h^2}\frac{r}{k}}-C_\beta \frac{k}{r}e^{-2rh}+\frac{C}{r^{m+\sigma}}+kO_k(e^{-2(1+\sigma)rh})\Bigr)\cr
&+O\Bigl(\frac{C(r)}{r^{m+\sigma}}+\frac{D(r)}{r^{n+\sigma}}\Bigr)+\frac{1}{k^{m+2+\sigma}},
\end{align*}
 where $A_0,\,A_1,\,A_2$ are constants in Proposition A.1 and $C(r),\,D(r)$ are functions independent of $h$ and can be absorbed in $O(1)$. We prove the case $m=n$, since the other cases are similar. If $m=n$, then
\begin{align*}
F_1(r,h)=A_0+\frac{A_1}{r^m}-2C_\beta e^{-2\pi\sqrt{1-h^2}\frac{r}{k}}-C_\beta e^{-2rh}.
\end{align*}
 Then we consider the system
\begin{eqnarray}\label{2.24}
\begin{cases}
F_{1r}(r,h)=-A_1\frac{m}{r^{m+1}}+4C_\beta\pi\sqrt{1-h^2}e^{-\frac{2\pi\sqrt{1-h^2}r}{k}}+2C_\beta he^{-2rh}=0,\cr
\,\,\,\,\,\cr
F_{1h}(r,h)=-4C_\beta\pi\frac{hr}{\sqrt{1-h^2}}\frac{e^{-2\pi\sqrt{1-h^2}\frac{r}{k}}}{k}+2C_\beta re^{-2rh}=0.
 \end{cases}
 \end{eqnarray}
From \eqref{2.24}, we can get
\begin{align*}
-A_1\frac{m}{r^{m+1}}+4B_1\pi\frac{e^{-2\pi\sqrt{1-h^2}\frac{r}{k}}}{k}\Bigl(\sqrt{1-h^2}+\frac{h^2}{\sqrt{1-h^2}}\Bigr)=0.
\end{align*}
 Define
 $$
 H(r,h)=e^{-2\pi\sqrt{1-h^2}\frac{r}{k}},\,\,\,\,\,\,\,\,\,G(r,h)=e^{-2rh}.
 $$
 Then \eqref{2.24} implies
 \begin{align*}
H(r,h)=\frac{A_1k\frac{m}{r^{m+1}}}{4B_1\pi (\sqrt{1-h^2}+\frac{h^2}{\sqrt{1-h^2}})},\,\,\,\,\,\,\,\,\,
G(r,h)=\frac{2\pi h}{\sqrt{1-h^2}k}e^{-2\pi\sqrt{1-h^2}\frac{r}{k}}.
 \end{align*}
 We define the space
 \begin{align*}
\bar{S}_k=\Bigl\{(r_k,h_k)\bigr|\Bigl[(\frac{m}{2\pi}-\frac{\beta}{100})k\ln k,(\frac{m}{2\pi}+\frac{\beta}{100})k\ln k\Bigr]\times\Bigl[\frac{\pi(m+2)}{m}-\frac{\alpha}{100}\frac{1}{k},\frac{\pi(m+2)}{m}+\frac{\alpha}{100}\frac{1}{k}\Bigr]\Bigr\},
 \end{align*}
and the mapping $T:\bar{S}_k\rightarrow\R^2$ satisfying $T\left(r_k,h_k\right)=\left(H\left(r_k,h_k\right),G\left(r_k,h_k\right)\right)$,
where $\alpha,\beta$ are constants. The system \eqref{eqs1.1} is equivalent to finding a fixed point of
\begin{align*}
(r,h)=&T^{-1}\Big(\frac{A_1k\frac{m}{r^{m+1}}}{4C_\beta
\pi(\sqrt{1-h^2}+\frac{h^2}{\sqrt{1-h^2}})},\frac{h}{2C_\beta\sqrt{1-h^2}}
\frac{A_1k\frac{m}{r^{m+1}}}{\sqrt{1-h^2}+\frac{h^2}{\sqrt{1-h^2}}}\Big)\cr
=&A(r,h)=:(a_1(r,h),a_2(r,h)).
\end{align*}
It is easy to show that
\begin{align*}
|a_1(r_1,h_1)-a_2(r_2,h_2)|+|a_2(r_1,h_1)-a_2(r_2,h_2)|\leq o_k(1)(|r_1-r_2|+|h_1-h_2|),
\end{align*}
for all $(r_1,h_1),(r_2,h_2)\in S_k$. By using the Contraction Mapping principle, we can prove that there exists a fixed point $(\bar{r}_k,\bar{h}_k)\in \bar{S}_k$.
Similar proof as \cite{DM}, we complete the proof.

\section{segregated vector solutions and the proof of Theorem \ref{th1.2}}
In this section, we consider synchronized vector solutions and prove Theorem \ref{th1.2} by proving Theorem \ref{th1.4}. Let
\begin{align}
\bar{M}_{1j}=\frac{\partial W_{\mu_1,\bar{x}_j}}{\partial r},\,\bar{M}_{2j}=\frac{\partial W_{\mu_1,\bar{x}_j}}{\partial h},\,\bar{N}_{1j}=\frac{\partial W_{\mu_2,\bar{y}_j}}{\partial \rho},\,\bar{N}_{1j}=\frac{\partial W_{\mu_2,\bar{y}_j}}{\partial h},
\end{align}
 \begin{align}
\underline{M}_{1j}=\frac{\partial W_{\mu_1,\bar{x}_j}}{\partial r},\,\underline{M}_{2j}=\frac{\partial W_{\mu_1,\bar{x}_j}}{\partial h},\,\underline{N}_{1j}=\frac{\partial W_{\mu_2, \bar{y}_j}}{\partial \rho},\,\underline{N}_{1j}=\frac{\partial W_{\mu_2,\bar{y}_j}}{\partial h},
\end{align}
where $\bar{x}_j$, $\bar{y}_j$ $(j=1,2,\cdots,k)$\,are defined in Section 2. Denote
\begin{align*}
W_{r,h}=\sum\limits_{j=1}^kW_{\mu_1,\bar{x}_j}+W_{\mu_1,\underline{x}_j},\,\,\,\,\,\,W_{\rho,h}=\sum\limits_{j=1}^kW_{\mu_2,\bar{y}_j}+W_{\mu_2,\underline{y}_j}.
\end{align*}
For simplicity of notations, in the sequel we use $W_{r,h},W_{\rho,h}$ to replace $U,\,V$ respectively. In this section, we assume
\begin{align*}
(r,\rho)\in S_k\times S_k=\Bigl[\,\bigl(\frac{m}{2\pi}-\delta\bigr)k\ln k,\bigl(\frac{m}{2\pi}+\delta\bigr)k\ln k\,\Bigr]\times\Bigl[\,\bigl(\frac{m}{2\pi}-\delta\bigr)k\ln k,\bigl(\frac{m}{2\pi}+\delta\bigr)k\ln k\,\Bigr],
\end{align*}
\begin{align*}
h\in \Bigl[\bigl(\frac{\pi(m+2)}{m}-\xi\bigr)\frac{1}{k},\bigl(\frac{\pi(m+2)}{m}+\xi\bigr)\frac{1}{k}\Bigr]=:H_0.
\end{align*}
Define
\begin{align*}
\bar{E}=\Bigl\{(u,v)\in H\times H, &\int_{\R^N}W_{\mu_1,\bar{x}_j}^2\bar{M}_{lj}u=0,\int_{\R^N}W_{\mu_1,\underline{x}_j}^2\underline{M}_{lj}u=0,\cr
&\,\,\,\,\,\,\,\,\int_{\R^N}W_{\mu_2,\bar{y}_j}^2\bar{N}_{lj}v=0,\int_{\R^N}W_{\mu_2,\bar{y}_j}^2\underline{N}_{lj}v=0,j=1,2\cdots,k,l=1,2\Bigr\},
\end{align*}
\begin{align*}
L^*(\phi,\xi)=&\int_{\R^N}\left[\left(|\nabla \phi|^2+P(|x|)\phi^2-3\mu_1 W_{r,h}^2\phi^2\right)+\left(|\nabla\xi|^2+Q(|x|)\xi^2-3\mu_2 W_{\rho,h}^2\xi^2\right)\right]\cr
&-\beta\int_{\R^3}\left(W_{r,h}^2\xi^2+4W_{r,h}W_{\rho, h}\phi\xi+W_{\rho,h}^2\phi^2\right),
\end{align*}

\begin{align*}
\ell^*(\phi,\xi)=&\int_{\R^3}(P(x)-1)W_{r,h}\phi+\mu_1\int_{\R^3}\Bigl(W_{r,h}^3-\sum\limits_{j=1}^kW_{\mu_1,\bar{x}_j}^3-\sum\limits_{j=1}^kW_{\mu_1, \underline{x}_j}^3\Bigr)\phi\cr
&+\int_{\R^3}(Q(x)-1)W_{\rho,h}\xi+\mu_2\int_{\R^3}\Bigl(W_{\rho,h}^3
-\sum\limits_{j=1}^kW^3_{\mu_2,\bar{y}_j}-\sum\limits_{j=1}^kW^3_{\mu_2,\underline{y}_j}\Bigr)\xi\cr
&-\beta\int_{\R^3}\Bigl(W_{r, h}W_{\rho,h}^2\phi+W_{r,h}^2W_{\rho,h}\xi\Bigr),
\end{align*}

\begin{align*}
R^*(\phi,\xi)=&-\frac{1}{4}\int_{\R^3}\left[\left(W_{r,h}+\phi\right)^4+\left(W_{\rho,h}+\xi\right)^4-W_{r,h}^4-W_{\rho,h}^4-4W_{r, h}^3\phi-4W_{\rho, h}^3\xi\right.\cr
&\left.-6\mu_1 W_{r,h}^2\phi^2-6\mu_2 W_{\rho,h}^2\xi^2\right]\cr
&-\frac{\beta}{2}\int_{\R^3}\left[\left(W_{r,h}+\phi\right)^2\left(W_{\rho, h}+\xi\right)^2-W_{r,h}^2W_{\rho,h}^2-2\left(W_{r,h}W_{\rho,h}^2\phi-W_{r,h}^2W_{\rho,h}\xi\right)\right.\cr
&\left.-2\left(W_{r,h}^2\xi^2+4W_{r,h}^2\xi^2+4W_{r,h}W_{\rho,h}\phi\xi+W_{\rho,h}^2\phi^2\right)\right].
\end{align*}
\begin{lemma}\label{lm3.1}
There is a constant $C>0$, independent of $k$ such that for any $(r,\rho)\in S_k\times S_k$,  there holds
\begin{align*}
\|L^*(u,v)\|\leq C\|(u,v)\|,\,\,\,(u,v)\in E.
\end{align*}
\end{lemma}
\begin{lemma}\label{lm3.2}
There exists $\rho^*>0$ independent of $k$, such that for any $(r,\rho)\in S_k\times S_k$ if $\beta<\beta^*$, then
\begin{align*}
\|L^*(u,v)\|\geq \rho^*\|(u,v)\|,\,\,\,(u,v)\in E.
\end{align*}
\end{lemma}
\begin{proof}
The argument is similar to the proof of Lemma 2.2. Arguing by contradiction. We suppose that there are $k\rightarrow+\infty$, $(r,\rho)\in S_k\times S_k$ and $(u_k,v_k)\in E$ with $\|(u_k,v_k)\|^2=k$, and
\begin{equation}\label{eqs3.3.1}
\bigl\langle L(u_k,v_k),(\phi,\xi)\bigr\rangle=o(1)\|(u_k,v_k)\|\|(\phi,\xi)\|,\,\,\,\forall\,(\phi,\xi)\in \bar{E}.
\end{equation}
For $j=1,2,\cdots,k$, let
\begin{align*}
\Omega_j^+=\Bigl\{z_3>0,\,\,\,\,\,\,\,\,\Bigl\langle\frac{z'}{|z'|},\frac{x'_j}{|x'_j|}\Bigr\rangle \geq \cos\frac{\pi}{k}\Bigr\},
\end{align*}
\begin{align*}
\Omega_j^{+*}=\Bigl\{z_3>0,\,\,\,\,\,\,\,\,\,\Bigl\langle\frac{z'}{|z'|},\frac{y'_j}{|y'_j|}\Bigr\rangle\geq\cos\frac{\pi}{k}\Bigr\}.
\end{align*}
We will use $r,\rho$ to replace $r_k,\rho_k$ respectively. By symmetry, we see from \eqref{eqs3.3.1},
\begin{align}\label{eqs03.1}
&\int_{\Omega_1}\left(\nabla u_k\nabla\phi+P\left(|x|\right)u_k\phi-3\mu_1 W_{r,h}u_k\phi\right)+\int_{\Omega_1}\left(\nabla v_k\nabla\xi+Q\left(|x|\right)v_k\xi-3\mu_2 W_{\rho,h}v_k\xi\right)\cr
&-\beta\int_{\Omega_1}\left(W_{\mu, h}^2v_k\xi+W_{\rho,h}^2u_k\phi+2 W_{r,h}W_{\rho,h}u_k\xi+2W_{r,h}W_{\rho, h}v_k\phi\right)\cr
=&\frac{1}{k}\bigl\langle L\left(u_k,v_k\right),\left(\phi,\xi\right)\bigr\rangle=o\left(1\right)\left(\frac{1}{\sqrt{k}}\|\left(\phi,\xi\right)\|\right),
\end{align}
and
\begin{align}\label{eqs03.2}
\int_{\Omega_1}\left(|\nabla u_k|^2+P(|x|)u_k^2+|\nabla v_k|^2+Q(|x|)v_k^2\right)=1.
\end{align}
Obviously, estimates \eqref{eqs03.1} and \eqref{eqs03.2} are also true on $\Omega_1^{*}$.

Let
$$
\bar{u}_k=u_k(x-\bar{x}^1),\,\,\,\,\,\,\bar{v}_k=v_k(x-\bar{y}^1).
$$
Now we consider $\bar{u}_k$ in details. The analysis on $\bar{v}_k$ is similar for $v_k$ and $u_k$ are also even with respect to the axis $\bar{y}^1$.

We may assume that there exists $\bar{u}\in H^1(\R^3)$, such that as $k\rightarrow+\infty$,
\begin{align*}
\bar{u}_k\rightarrow\bar{u},\,\,\hbox{weakly \,in}\,\, H_{loc}^1(\R^3),\,\,\,\,\,\,\,\bar{u}_k\rightarrow\bar{u},\,\,\hbox{strongly\, in}\, L_{loc}^2(\R^3).
\end{align*}
Let $\bar{\phi}\in C_0^\infty(B_R(0))$ and be even in $x_h,\,h=2,3$. Define $\bar{\phi}_k=:\bar{\phi}(x-\bar{x}^1)\in C_0^\infty(B_R(\bar{x}^1))$. Then choosing $(\phi,\xi)=(\bar{\phi}_k,0)$
in \eqref{eqs03.1} and proceeding as we did in Lemma \ref{lm2.2}, we can see that $\bar{u}$ satisfies
\begin{align*}
-\Delta \bar{u}+\bar{u}-3\mu_1W_{\mu_1}^2\bar{u}=0,\,\,\,\hbox{in}\,\,\R^3.
\end{align*}
Also, by the non degeneracy of $U_{\mu_1}$, we find $\bar{u}=0$.

Using the same argument on $\bar{\Omega}_1^+$, we can prove that as $k\rightarrow+\infty$,
\begin{align*}
 \bar{v}_k\rightarrow0,\,\,\,\,\hbox{weakly in}\, H_{loc}^1(\R^3),\,\,\,\, \bar{v}_k\rightarrow0,\,\,\,\,\hbox{strongly in}\, L_{loc}^2(\R^3).
\end{align*}
As a result, it holds
\begin{align*}
\int_{B_R(\bar{x}_1)}u_k^2=o(1),\,\,\,\int_{B_R(\bar{y}_1)}v_k^2=o(1),\,\,\,\,\,\forall\,\,R>0.
\end{align*}
On the other hand, using Lemmas \ref{lmA.1}-\ref{lmA.2}, we obtain
\begin{align*}
W_{r,h}\leq Ce^{-\frac{|x-\bar{x}_1|}{2}},\,\,\,\,x\in \Omega^+;\,\,\,\,\,\, W_{\rho,h}\leq Ce^{-\frac{|x-\bar{y}_1|}{2}},\,\,x\in \tilde{\Omega}^+_1.
\end{align*}
Thus, from \eqref{eqs03.1}, we see
\begin{align}\label{eqs03.3}
o(1)k=&\int_{\R^3}\left(|\nabla u_k|^2+P(|x|)u_k^2-3\mu_1 W_{r,h}^2u_k^2\right)+\int_{\R^3}\left(|\nabla v_k|^2+Q(|x|)v_k^2-3\mu_2W_{\rho,h}^2v_k^2\right)\cr
&-\beta\int_{\R^3}\left(W_{\rho,h}^2u_k^2+4W_{r,h}W_{\rho,h}u_kv_k+W_{r,h}^2v_k^2\right)\cr
=&\int_{\R^3}\left(|\nabla u_k|^2+P(|x|)u_k^2+|\nabla v_k|^2+Q(x)v_k^2\right)-\beta\int_{\R^3}\left(W_{\rho,h}^2u_k^2+W_{r,h}^2v_k^2\right)\cr
&-3\int_{\Omega_1}\mu_1 W_{r,h}^2u_k^2-3\mu_2\int_{\tilde{\Omega}_1}W_{\rho,h}^2v_k^2-4\beta\int_{\Omega_1}W_{r,h}W_{\rho,h}u_kv_k\cr
=&k-\beta\int_{\R^3}\left(U^2_{r,h}v_k^2+V_{\rho,h}^2u_k^2\right)+ko(e^{-\frac{\pi}{k}r}+o(1)+o(e^{-R}))\cr
\geq&k-C\beta k+ko(e^{-\frac{\pi r}{k}}+o(1)+o(e^{-R})).
\end{align}
Note that
\begin{align*}
0\leq\int_{\R^3}W_{r,h}^2v_k^2+W_{\rho,h}u_k^2\leq C\int_{\R^3}(u_k^2+v_k^2)\leq Ck,
\end{align*}
where $C$ is independent of $k$. If we choose $\beta< \beta^*=\frac{1}{C}$, then \eqref{eqs03.3} is impossible for large $R$ and $k$. Consequently, we complete the proof.
\end{proof}

Now we apply the above reduction process to the functional $J(\bar{\phi},\bar{\xi})$.

\begin{proposition}
There is an integer $k_0>0$, such that for each $k\geq k_0$, there is a $C^1$ map from $S_k\times S_k$ to $H_P\times H_Q$:$(\bar{\phi},\bar{\xi})=(\bar{\phi}(r,\rho),\bar{\xi}(r,\rho))$, $r=|x^1|,\rho=|y^1|$ satisfying $(\bar{\phi},\bar{\xi})\in E$ and
\begin{align*}
\Bigl\langle\frac{\partial J(\bar{\phi},\bar{\xi})}{\partial(\bar{\phi},\bar{\xi})},(h,g)\Bigr\rangle=0,\,\,\,\forall\,\,(h,g)\in \bar{E}.
\end{align*}
Moreover, there holds
\begin{align*}
\|(\bar{\phi},\bar{\xi})\|\leq Ck\Big(\frac{1}{r^m}+\frac{1}{\rho^m}+k^{\frac{1}{2}}\frac{e^{-\frac{2\pi r}{k}}}{r}+k^{\frac{1}{2}}\frac{e^{-\frac{2\pi \rho}{k}}}{\rho}+|\beta|k^{\frac{1}{2}}
\frac{e^{-\sqrt{(1-h^2)((\rho-r\cos\frac{\pi}{k})^2
+r^2(\frac{\pi}{k})^2}}}{r}\Big).
\end{align*}
\end{proposition}
\begin{proof}
We see that $\bar{\ell}(\bar{\phi},\bar{\xi})$ is a bounded linear functional in $E$. Thus, there is $\bar{\ell}_k\in E$ such that $$\bar{\ell}(\bar{\phi},\bar{\xi})=\bigl\langle\bar{\ell}_k,(\bar{\phi},\bar{\xi})\bigr\rangle.$$
Hence, we only need to use the argument as the proof of Proposition 2.5 and the following estimate on $\|\bar{\ell}_k\|$
\begin{align}\label{eqs33}
\|\bar{\ell}_k\|\leq Ck\Big(\frac{1}{r^m}+\frac{1}{\rho^m}+k^{\frac{1}{2}}\frac{e^{-\frac{2\pi r}{k}}}{r}+k^{\frac{1}{2}}\frac{e^{-\frac{2\pi \rho}{k}}}{\rho}+|\beta|k^{\frac{1}{2}}
\frac{e^{-\sqrt{(1-h^2)((\rho-r\cos\frac{\pi}{k})^2
+r^2(\frac{\pi}{k})^2)}}}{r}\Big),
\end{align}
where $C$ is independent of $k$ and $\beta$.

Now we prove \eqref{eqs33}. Indeed, since in Lemma 2.6, we have a similar estimate on the first four terms of ${\bar{\ell}(\bar{\phi},\bar{\xi})}$, we only need to estimate $\ds\int_{\R^3}\left(W_{r,h}W_{\rho,h}^2\bar{\phi}+W_{r,h}^2W_{\rho,h}\bar{\xi}\right).$

By symmetry, we see
\begin{align*}
&\int_{\R^3}W_{r,h}W_{\rho,h}^2\bar{\phi}=k\int_{\Omega_1}W_{r,h}W_{\rho,h}\bar{\phi}\cr
=&k\int_{\Omega_1^+}\Bigl(W_{\mu_1,\bar{x}_1}W^2_{\mu_2,\bar{y}_1}+W_{\mu_1,\bar{x}_1}\sum\limits_{j=2}^{k-1}W^2_{\mu_2,\bar{y}_j}
+W_{\mu_1,\bar{x}_1}W_{\mu_2,\bar{y}_k}^2
+W^2_{\mu_2,\bar{y}_1}\sum\limits_{j=2}^{k-1}W_{\mu_1,\bar{x}_1}\cr
&+\sum\limits_{j=2}^kW_{\mu_1,\bar{x}_j}\sum\limits_{j=2}^kW_{\mu_2,\bar{y}_j}\Bigr)\bar{\phi}\cr
&+k\int_{\Omega_1^-}\Bigl(W_{\mu_1,\underline{x}_1}W^2_{\mu_2,\underline{y}_1}+W_{\mu_1,\underline{x}_1}\sum\limits_{j=2}^{k-1}W^2_{\mu_2,\underline{y}_j}
+W_{\mu_1,\underline{x}_1}W_{\mu_2,\underline{y}_k}^2+W^2_{\mu_2,\underline{y}_1}\sum\limits_{j=2}^{k-1}W_{\mu_1,\underline{x}_1}\cr
&+\sum\limits_{j=2}^kW_{\mu_1,\underline{x}_j}\sum\limits_{j=2}^kW_{\mu_2,\underline{y}_j}\Bigr)\bar{\phi}\cr
\leq &C\frac{e^{-|\bar{x}_1-\bar{y}_1|}}{|\bar{x}_1-\bar{y}_1|}(\int_{\Omega}|\bar{\phi}|^2dx)^{\frac{1}{2}}
\leq Ck^{\frac{1}{2}}\frac{k}{r}e^{-\sqrt{(1-h^2)((\rho-r \cos\frac{\pi}{k})^2+r^2\sin^2\frac{\pi}{k})}}\|\bar{\phi}\|^2.
\end{align*}
Hence, we complete the proof.
\end{proof}

Finally, we will prove Theorem 1.4.

{\bf Proof of Theorem 1.4 } Let $ (\varphi_{r, \rho},\psi_{r, \rho})$ be the mapping obtained in Proposition 3.3. Define
$$
F(r, \rho)= I(W_{r,h}+\bar{\varphi}_k,W_{\rho,h}+\bar{\psi}_k), \quad \forall\, (r,\rho)\in S_k\times S_k,\,h\in H_0.
$$
With similar arguments used in Proposition 3 of \cite{OR} (see also \cite{CNY}), we can check that for $k$ sufficiently large, if $(r,\rho) \in S_k\times S_k$ is a critical point of $F$, then $(W_{r,h},W_{\rho,h})\in H$ is a critical point of $I$.

\medskip
Consider first the case $(ii)$ of $(H_m)$. For simplicity, denote $m = m_i$. Applying Lemma \ref{lm2.1}, Lemma \ref{lm2.4} and  Proposition \ref{proA.4}, for $k$ large enough and any $(r, \rho) \in S_k\times S_k$, we have
\begin{align*}
F(r,\rho)
&= I(W_{r,h},W_{\rho,h})+\ell^*(\bar{\varphi}_k,\bar{\psi}_k)+\frac{1}{2}\langle L^*(\varphi_k,\psi_k),(\varphi_k,\psi_k)\rangle+R^*(\varphi_k,\psi_k)
\\
&=
I(W_{r,h},W_{\rho,h})+O\Bigl(\|\ell^*\|\|(\bar{\varphi}_k,\bar{\psi}_k)\|+\|(\bar{\varphi}_k,\bar{\psi}_k)\|^2\Bigr)
\\
&=
k\Bigl(B_0+\frac{B_1}{r^m}+\frac{B_2}{\rho^n}-C_1\frac{k}{r}e^{-2\pi\sqrt{1-h^2}\frac{r}{k}}-\frac{k}{r}D_1e^{-2rh}-C_2\frac{k}{\rho}e^{-2\pi\sqrt{1-h^2}\frac{\rho}{k}}-D_2\frac{k}{\rho}e^{-2\rho h}
+O(\frac{1}{r^{2m}})\Bigr).
\end{align*}
Using the same proceeding as we did in proof of Theorem \ref{th1.3}, we claim that the maximal value of $F$ over $S_k\times S_k$ is attained by some interior point of $S_k\times S_k$.
\section{the sketch of proof for Theorem 1.5}
Since the approach is very similar to that for Theorem \ref{th1.2}, we omit the details and just explain the main difference. For any positive even number $2l$, set
\begin{align*}
\bar{W}_{r,h}=\sum\limits_{j=1}^{2l}(-1)^j(W_{\mu_1,\bar{x}_j}
+W_{\mu_1,\underline{x}_j}),\,\,\bar{W}_{\rho,h}
=\sum\limits_{j=1}^{2l}(-1)^j(W_{\mu_1,\bar{y}_j}
+W_{\mu_1,\underline{y}_j}).
\end{align*}
We will find a solution for system \eqref{eqs1.1} of the form $(\bar{U}_{r,h}+\varphi^*,\bar{V}_{r,h}+\psi^*)$. Notice also that the number of peaks is now $2l$ instead of $l$, essentially, $I_3,\,I_4$ and $I_5$ will have different form comparing to the proof of Propositions \ref{proA.4} and \ref{proA.5}. For example, here we have
\begin{align*}
&-\frac{1}{4}\int_{\R^3}W_{r,h}^4-\sum\limits_{k=1}^{2l}W_{\mu_1,x_k}^4-2\sum\limits_{i\neq k}^{2l}(-1)^{k+i}W_{\mu_1x_k}W_{\mu_1x_k}\cr
=&-\frac{\mu_1 l}{2}\int_{\Omega_1}W_{\mu_1,x_1}(-1)^{k+1}\sum\limits_{k=2}^{2l}W_{\mu_1,x_k}=\frac{\mu l}{2}\int_{\Omega_1}W_{\mu_1,x_1}(-1)^k\sum\limits_{k=2}^{2l}W_{\mu_1,x_k}\cr
=&C\frac{l^2}{r}e^{-\frac{2\pi r\sqrt{1-h^2}}{l}}+O(ke^{-\frac{3\pi r}{l}})
\end{align*}
The search of a critical point with the form $(\bar{W}_{r,h}+\varphi^*,\bar{W}_{\rho,h}+\psi^*)$ will be reduced to search a critical point of the following function in the interior of $S_k\times S_k$:
\begin{align*}
I(\bar{W}_{r,h},\bar{W}_{\rho,h})=&k\Bigl(B_0+\frac{B_1}{r^m}+\frac{B_2}{\rho^n}-\tilde{C}_1\frac{k}{r}e^{-2\pi\sqrt{1-h^2}\frac{r}{k}}-\frac{k}{r}\tilde{D}_1e^{-2rh}-\tilde{C}_2\frac{k}{\rho}e^{-2\pi\sqrt{1-h^2}\frac{\rho}{k}}-\tilde{D}_2\frac{k}{\rho}e^{-2\rho h}\cr
&+o(1)\beta\frac{k}{r}e^{-\sqrt{(1-h^2)((\rho-r\cos\frac{\pi}{k})^2+r^2(\frac{\pi}{k})^2)}}\Bigr),
\end{align*}
where $B_0,\,B_1,\,B_2\,\tilde{C}_1,\,\tilde{D}_1,\,\tilde{C}_2,\,\tilde{D}_2$ are positive constants in Proposition \ref{proA.5}. The rest of the proof can be finished as in the proof of Theorem \ref{th1.4}.
\appendix
\section{Energy expansion}
In this section, we mainly give some energy expansions. Recall
\begin{align*}
\bar{x}_j=r\Bigl(\sqrt{1-h^2}\cos\frac{2(j-1)\pi}{k},\sqrt{1-h^2}\sin\frac{2(j-1)\pi}{k},h\Bigr),j=1,2,\cdots,k,
\end{align*}
\begin{align*}
\underline{x}_j=r\Bigl(\sqrt{1-h^2}\cos\frac{2(j-1)\pi}{k},\sqrt{1-h^2}\sin\frac{2(j-1)\pi}{k},-h\Bigr),j=1,2,\cdots,k.
\end{align*}
The parameters $h$ and $r$ are positive numbers and are chosen in the range
\begin{align*}
h\in[\alpha_0\frac{1}{k},\alpha_1\frac{1}{k}],\,\,\,\,\,\,\,\,r\in (C_1k\ln k,C_2k\ln k),
\end{align*}
\begin{align*}
\Omega_j=\Bigl\{y=(y_1,y_2,y_3)\in\R^3:\langle\frac{(y_1,y_2)}{|(y_1,y_2)|},
(\cos\frac{2(j-1)\pi}{k},\sin\frac{2(j-1)\pi}{k})
\rangle\geq\cos\frac{\pi}{k}\Bigr\}.
\end{align*}
\begin{lemma}\label{lmA.1}(see \cite{DM} Lemma 3.1)
For $r,\,h$ being the parameters in $S_k\times D_h$ and any $\eta\in(0,1]$, there is a constant $C>0$ such that
\begin{align*}
&\sum\limits_{j=2}^kW_{\bar{x}_j}(y)\leq Ce^{-\eta\sqrt{1-h^2}r\frac{k}{r}}e^{-(1-\eta|y-\bar{x}_1|)}\,\hbox{for\,all\,}\,y\in \Omega_1^+,\cr
&\sum\limits_{j=2}^kW_{\overline{x}_j}(y)\leq Ce^{-\eta\sqrt{1-h^2}r\frac{k}{r}}e^{-(1-\eta|y-\overline{x}_1|)}\,\hbox{for\,all\,}\,y\in \Omega_1^+,\cr
&W_{\overline{x}_j}(y)\leq Ce^{-\eta hr}e^{-(1-\eta)|y-\bar{x}_1|}\,\,\,\hbox{for all}\,\,\,\,y\in \Omega_1^+.
\end{align*}
\end{lemma}

\begin{lemma}\label{lmA.2}(see \cite{DM} Lemma 3.2)
For $i=1,2\cdots,k$, there exist some small constant $\sigma>0$ such that the following expansions hold
\begin{align*}
&\int_{\R^3}W_{\bar{x}_1}^3W_{\bar{x}_i}
=Ce^{-|\bar{x}_1-\bar{x}_i|}+O_k(e^{-(1+\sigma)|\bar{x}_1-\bar{x}_i|}),\cr
&\int_{\R^3}W_{\underline{x}_i}^3
W_{\bar{x}_1}=Ce^{-|\bar{x}_1-\overline{x}_i|}
+O_k(e^{-(1+\sigma)|\bar{x}_1-\overline{x}_i|}),\cr
&\int_{\R^3}W_{\underline{x}_1}^3W_{\bar{x}_1}
=Ce^{-2rh}+O_k(e^{-2(1+\sigma)rh}),
\end{align*}
where $C>0$ is a constant.
\end{lemma}

\begin{lemma}\label{lmA.3}
There exists some small constant $\sigma>0$ such that the following expansions hold
\begin{align*}
&\sum\limits_{i=2}^k\int_{\R^3}W_{\bar{x}_1}^3W_{\bar{x}_i}
=2Ce^{-2\pi\sqrt{1-h^2}\frac{r}{k}}+O_k(e^{-2(1+\sigma)
\pi\sqrt{1-h^2}\frac{r}{k}}),\cr
&\sum\limits_{j=1}^k\int_{\R^3}W_{\overline{x}_j}^3W_{\bar{x}_1}=Ce^{-2rh}+O_k(e^{-2(1+\sigma)rh}).
\end{align*}
\end{lemma}
\begin{proposition}\label{proA.4}
For all $(r,h)\in S_k$, there exists some constant $\sigma>0$ such that
\begin{align*}
I\big(U_{r,h},V_{r,h}\big)=k\Big(A_0+\frac{aA_1}{r^m}+\frac{bA_2}{r^n}-2C_\beta \frac{k}{r}e^{-2\pi\sqrt{1-h^2}\frac{r}{k}}-D_\beta \frac{k}{r}e^{-2rh}+\frac{C}{r^{m+\sigma}}+kO_k(e^{-2(1+\sigma)rh})\Big),
\end{align*}
where $A_0=\frac{\mu_1+\mu_2-2\beta}{2(\mu_1\mu_2-\beta^2)}\ds\int_{\R^3}W^4dx$, $A_1=\ds\frac{\alpha^2}{2}\int_{\R^3}W^2dx$, $A_2=\ds\frac{\gamma^2}{2}\int_{\R^3}W^2dx$, $C_\beta,\,D_\beta$ are positive constants.
\end{proposition}
\begin{proof}
By direct computations, we have
\begin{align}\label{eqs3.1}
I(U_{r,h},V_{r,h})=&\frac{1}{2}\int_{\R^3}\left(|\nabla U_{r,h}|^2+P(|x|)U_{r,h}^2+|\nabla V_{r,h}|^2+Q(|x|)V_{r,h}^2\right)
-\frac{1}{4}\int_{\R^3}\left(\mu_1|U_{r,h}|^4+\mu_2|V_{r,h}|^4\right)\cr
&-\frac{\beta}{2}\int_{\R^3}U_{r,h}^2V_{rh}^2\cr
=&\frac{\mu_1}{4}k\int_{\R^3}(U_{\bar{x}_i}^4+U_{\underline{x}_i}^4)
+\frac{\mu_2}{4}k\int_{\R^3}(V_{\bar{x}_i}^4+V_{\underline{x}_i}^4)
+\frac{\beta}{2}k\int_{\R^3}(U_{\bar{x}_i}^2V_{\bar{x}_i}^2+U_{\underline{x}_i}^2V_{\underline{x}_i}^2)\cr
&+\frac{1}{2}\int_{\R^3}(P(x)-1)U_{r,h}^2+(Q(x)-1)V_{r,h}^2\cr
&-\frac{\mu_1}{4}\int_{\R^3}\Bigl[(\sum\limits_{j=1}^kU_{\bar{x}_j}+\sum\limits_{j=1}^kU_{\underline{x}_j})^4-\sum\limits_{j=1}^kU_{\bar{x}_j}^4
-\sum\limits_{j=1}^kU_{\underline{x}_j}^4-2\sum\limits_{i\neq j}U_{\bar{x}_j}^3U_{\bar{x}_i}\cr
&-2\sum\limits_{i,j}U_{\bar{x}_j}^3U_{\underline{x}_i}-2\sum\limits_{i\neq j}U_{\underline{x}_j}^3U_{\underline{x}_i}-2\sum\limits_{i,j}U_{\underline{x}_j}^3U_{\bar{x}_i}\Bigr]\cr
&-\frac{\mu_2}{4}\int_{\R^3}\Bigl[(\sum\limits_{j=1}^kV_{\bar{x}_j}+\sum\limits_{j=1}^kV_{\underline{x}_j})^4-\sum\limits_{j=1}^kV_{\bar{x}_j}^4
-\sum\limits_{j=1}^kV_{\underline{x}_j}^4-2\sum\limits_{i\neq j}V_{\bar{x}_j}^3V_{\bar{x}_i}\cr
&-2\sum\limits_{i,j}V_{\bar{x}_j}^3V_{\underline{x}_i}-2\sum\limits_{i\neq j}V_{\underline{x}_j}^3V_{\underline{x}_i}-2\sum\limits_{i,j}V_{\underline{x}_j}^3V_{\bar{x}_i}\Bigr]\cr
&-\frac{\beta}{2}\sum\limits_{i\neq j}\int_{\R^3}\Bigl(U_{r,h}^2V_{r,h}^2-\sum\limits_{i=1}^kU_{\bar{x}_i}^2V_{\bar{x}_i}^2-\sum\limits_{i=1}^kU_{\underline{x}_i}^2V_{\underline{x}_i}^2
-\sum\limits_{i\neq j}V_{\bar{x}_j}^2U_{\bar{x}_j}U_{\bar{x}_i}\cr
&-\sum\limits_{i,j}V_{\bar{x}_j}^2U_{\bar{x}_j}U_{\underline{x}_i}-\sum\limits_{i\neq j}V_{\underline{x}_j}^2U_{\underline{x}_j}U_{\underline{x}_i}-\sum\limits_{i,j}V_{\underline{x}_j}^2U_{\underline{x}_j}U_{\bar{x}_i}
-\sum\limits_{i\neq j}U_{\bar{x}_j}^2V_{\bar{x}_j}V_{\bar{x}_i}\cr
&-\sum\limits_{i,j}U_{\bar{x}_j}^2V_{\bar{x}_j}V_{\underline{x}_i}-\sum\limits_{i\neq j}U_{\underline{x}_j}^2V_{\underline{x}_j}V_{\underline{x}_i}-\sum\limits_{i,j}U_{\underline{x}_j}^2V_{\underline{x}_j}V_{\bar{x}_i}\Bigr)\cr
=&I_1+I_2+I_3+I_4+I_5.
\end{align}
Using symmetry and the asymptotic expressions of $P(|x|), Q(|x|)$, we get
\begin{align}\label{eqs3.2}
\frac{1}{2}\int_{\R^3}(P(x)-1)U_{r,h}^2=&k\int_{\Omega_1^+}(P(x)-1)\Bigl(U_{\underline{x}_1}+U_{\bar{x}_1}+\sum\limits_{j=2}^k U_{\underline{x}_j}+\sum\limits_{j=2}^k U_{\bar{x}_j}\Bigr)^2\cr
=&k\int_{\Omega_1^+}(P(x)-1)\Bigl[U_{\bar{x}_1}^2+O(U_{\bar{x}_1}U_{\underline{x}_1}+U_{\bar{x}_1}\sum\limits_{j=2}^kU_{\underline{x}_j})\Bigr]\cr
=&k\int_{\R^3}(P(x)-1)U_{\bar{x}_1}^2+k\int_{\R^3\setminus\Omega_1^+}(P(x)-1)U_{\bar{x}_1}^2\cr
&+kO_k\Bigl[\int_{\Omega_1^+}(P(x)-1)\bigl(U_{\bar{x}_1}^\tau U_{\underline{x}_1}e^{-(1-\tau)|\bar{x}_1-\underline{x}_1|}+U_{\bar{x}_1}^{2\tau}\sum\limits_{j=2}^ke^{-(1-\tau)|\bar{x}_1-\bar{x}_j|}\bigl)\Bigr]\cr
=&k\Bigl(\frac{A_1}{r^m}+O(\frac{C}{r^{m+\sigma}})\Bigr),
\end{align}
where $A_1=\ds\frac{\alpha^2}{2}\int_{\R^3}W^2dx$ and $\tau>0$ is a small constant.

By direct computations, we have
\begin{align}\label{eqs3.3}
I_3=&\int_{\R^3}\Bigl[(\sum\limits_{j=1}^kU_{\bar{x}_j}+\sum\limits_{j=1}^kU_{\underline{x}_j})^4-\sum\limits_{j=1}^kU_{\bar{x}_j}^4
-\sum\limits_{j=1}^kU_{\underline{x}_j}^4
-2\sum\limits_{i\neq j}U_{\bar{x}_j}^3U_{\bar{x}_i}-2\sum\limits_{i,j}U_{\bar{x}_j}^3U_{\underline{x}_i}\cr
&-2\sum\limits_{i\neq j}U_{\underline{x}_j}^3U_{\underline{x}_i}-2\sum\limits_{i,j}U_{\underline{x}_j}^3U_{\bar{x}_i}\Bigr]\cr
=&2k\int_{\Omega_1^+}\Bigl[(U_{\bar{x}_1}+\sum\limits_{i=2}^kU_{\bar{x}_i}+U_{\underline{x}_1}+\sum\limits_{i=2}^kU_{\bar{x}_i})^4
-(U_{\bar{x}_1}^4+\sum\limits_{i=2}^kU_{\bar{x}_i}^4)-(U_{\underline{x}_1}^4+\sum\limits_{i=2}^kU_{\underline{x}_i}^4)\cr
&-2(U_{\bar{x}_1}^3\sum\limits_{i=2}^kU_{\bar{x}_i}+U_{\bar{x}_1}\sum\limits_{i=2}^kU_{\bar{x}_i}^3+\sum\limits_{i\neq j}U_{\bar{x}_i}^3U_{\bar{x}_j})
-2(U_{\underline{x}_1}^3\sum\limits_{i=2}^kU_{\underline{x}_i}+U_{\underline{x}_1}\sum\limits_{i=2}^kU_{\underline{x}_i}^3+\sum\limits_{i\neq j}U_{\underline{x}_i}^3U_{\underline{x}_j})\cr
&-2(U_{\bar{x}_1}^3\sum\limits_{i\neq1}U_{\underline{x}_i}+\sum\limits_{j\neq1}U_{\bar{x}_j}^3U_{\underline{x}_i})
-2(U_{\underline{x}_1}^3\sum\limits_{i\neq1}U_{\bar{x}_i}+\sum\limits_{j\neq1}U_{\underline{x}_j}^3U_{\bar{x}_i})\Bigl]\cr
=&2k\int_{\Omega_1^+}\Bigl[U_{\bar{x}_1}^3\sum\limits_{i=2}^kU_{\bar{x}_i}+U_{\bar{x}_1}(\sum\limits_{i=2}^kU_{\underline{x}_i}+U_{\underline{x}_1})\Bigr]+O(ke^{-\frac{2\pi r}{k}}).
\end{align}
Similarly, we can estimate
\begin{align}\label{eqs3.4}
I_4=&-\frac{\mu_2}{4}\int_{\R^3}\Bigr[(\sum\limits_{j=1}^kV_{\bar{x}_j}+\sum\limits_{j=1}^kV_{\underline{x}_j})^4-\sum\limits_{j=1}^kV_{\bar{x}_j}^4
-\sum\limits_{j=1}^kV_{\underline{x}_j}^4
-2\sum\limits_{i\neq j}V_{\bar{x}_j}^3V_{\bar{x}_i}-2\sum\limits_{i,j}V_{\bar{x}_j}^3V_{\underline{x}_i}\cr
&-2\sum\limits_{i\neq j}V_{\bar{x}_j}^3V_{\bar{x}_i}-2\sum\limits_{i,j}V_{\bar{x}_j}^3V_{\underline{x}_i}\Bigr]\cr
=&2k\int_{\Omega_1^+}\Bigl[V_{\bar{x}_1}\sum\limits_{i=2}^kU_{\bar{x}_i}+V_{\bar{x}_1}^3(\sum\limits_{i=2}^kV_{\underline{x}_i}+V_{\underline{x}_1})\Bigr]+O(ke^{-\frac{3\pi r}{k}}),
\end{align}

\begin{align}\label{eqs3.5}
I_5=&-\frac{\beta}{2}\sum\limits_{i\neq j}\int_{\R^3}\Bigl[U_{r,h}^2V_{r,h}^2-\sum\limits_{i=1}^kU_{\bar{x}_i}^2V_{\bar{x}_i}^2-\sum\limits_{i=1}^kU_{\underline{x}_i}^2V_{\underline{x}_i}^2
-\sum\limits_{i\neq j}V_{\bar{x}_j}^2U_{\bar{x}_j}U_{\bar{x}_i}-\sum\limits_{i,j}V_{\bar{x}_j}^2U_{\bar{x}_j}U_{\underline{x}_i}\cr
&-\sum\limits_{i\neq j}V_{\underline{x}_j}^2U_{\underline{x}_j}U_{\underline{x}_i}-\sum\limits_{i,j}V_{\underline{x}_j}^2U_{\underline{x}_j}U_{\bar{x}_i}
-\sum\limits_{i\neq j}U_{\bar{x}_j}^2V_{\bar{x}_j}V_{\bar{x}_i}-\sum\limits_{i,j}U_{\bar{x}_j}^2V_{\bar{x}_j}V_{\underline{x}_i}\cr
&
-\sum\limits_{i\neq j}U_{\underline{x}_j}^2V_{\underline{x}_j}V_{\underline{x}_i}-\sum\limits_{i,j}U_{\underline{x}_j}^2V_{\underline{x}_j}V_{\bar{x}_i}\Bigr]\cr
=&2k\int_{\Omega_1^+}\Bigl[V_{\bar{x}_1}^2U_{\bar{x}_1}\sum\limits_{j=2}^kU_{\bar{x}_j}+V_{\bar{x}_1}^2U_{\bar{x}_1}\sum\limits_{j=2}^kU_{\underline{x}_j}
+U_{\bar{x}_1}^2V_{\bar{x}_1}\sum\limits_{j=2}^kV_{\bar{x}_j}+U_{\bar{x}_1}^2V_{\bar{x}_1}\sum\limits_{j=2}^kV_{\underline{x}_j}\Bigr].
\end{align}
Since $U_{x_i}=\frac{\alpha}{\gamma}V_{x^i}$,
combining \eqref{eqs3.1}-\eqref{eqs3.5}, we can get
\begin{align*}
I(U_{r,h},V_{r,h})=k(A_0+\frac{A_1}{r^m}+\frac{A_2}{r^n})-\sum\limits_{i=2}^k\int_{\R^3}(kC_{1\beta}U_{\bar{x}_1}^3U_{\bar{x}_i}-kD_{1\beta}\sum\limits_{j=1}^kU_{\underline{x}_j}^3U_{\bar{x}_1})
+O_k(\frac{1}{r^{m+\sigma}}).
\end{align*}
By Lemmas \ref{lmA.3}, we have
\begin{align*}
I(U_{r,h},V_{r,h})=k\Bigl(A_0+\frac{A_1}{r^m}+\frac{A_2}{r^n}-2C_\beta \frac{k}{r}e^{-2\pi\sqrt{1-h^2}\frac{r}{k}}-D_\beta \frac{k}{r}e^{-2rh}+\frac{C}{r^{m+\sigma}}+kO_k(e^{-2(1+\sigma)rh})\Bigr).
\end{align*}
Hence, we complete the proof.
\end{proof}

\begin{proposition}\label{proA.5}
For all $(r,\rho,h)\in S_k\times D$, there exists some small constant $\sigma>0$,
\begin{align*}
I(W_{r,h},W_{\rho,h})=&k\Bigl(B_0+\frac{B_1}{r^m}+\frac{B_2}{\rho^n}-C_1\frac{k}{r}e^{-2\pi\sqrt{1-h^2}\frac{r}{k}}-\frac{k}{r}D_1e^{-2rh}-C_2\frac{k}{\rho}e^{-2\pi\sqrt{1-h^2}\frac{\rho}{k}}-D_2\frac{k}{\rho}e^{-2\rho h}\cr
&+o(1)\beta\frac{k}{r}e^{-\sqrt{(1-h^2)((\rho-r\cos\frac{\pi}{k})^2+r^2(\frac{\pi}{k})^2)}}\Bigr),
\end{align*}
where $B_0=\frac{1}{2}\ds\int_{\R^3}\mu_1W_{\mu_1}^4+\mu_2W_{\mu_2}^4$, $B_1=\ds\int_{\R^3}W_{\mu_1}^2dx$, $B_2=\ds\int_{\R^3}W_{\mu_2}^2dx$, $C_1,\,D_1,\,C_2,\,D_2$\,are positive constants.
\end{proposition}
\begin{proof}
First, we have
\begin{align}\label{eqs3.6}
I(W_{r,h},W_{\rho,h})=&\frac{1}{2}\int_{\R^3}(|\nabla W_{r,h}|^2+P(|x|)W_{r,h}^2+|\nabla W_{\rho,h}|^2+Q(|x|)W_{\rho,h}^2)
-\frac{1}{4}\int_{\R^3}(\mu_1W_{r,h}^4+\mu_2W_{\rho,h}^4)\cr
&-\frac{\beta}{2}\int_{\R^3}W_{r,h}^2W_{\rho,h}^2\cr
=&\frac{k}{2}\int_{\R^3}\Bigl(\mu_1W_{\mu_1}^4+\mu_2W_{\mu_2}^4\Bigr)-\frac{\beta}{2}\int_{\R^3}W_{r,h}^2W_{\rho,h}^2
+\frac{1}{2}\int_{\R^3}\Bigl[P(|x|-1)W_{r,h}^2+(Q(x)-1)W_{\rho,h}^2\Bigr]\cr
&-\frac{\mu_1}{4}\int_{\R^3}\Bigl[\bigl(\sum\limits_{j=1}^kW_{\mu_1,\bar{x}_j}+\sum\limits_{j=1}^kW_{\mu_1,\underline{x}_j}\bigr)^4-\sum\limits_{j=1}^kW_{\mu_1,\bar{x}_j}^4
-\sum\limits_{j=1}^kW_{\mu_1,\underline{x}_j}^4
-2\sum\limits_{i\neq j}W_{\mu_1,\bar{x}_j}^3W_{\mu_1,\bar{x}_i}\cr
&-2\sum\limits_{i,j}W_{\mu_1,\bar{x}_j}^3W_{\mu_1,\underline{x}_i}
-2\sum\limits_{i\neq j}W_{\mu_1,\underline{x}_j}^3W_{\mu_1,\underline{x}_i}
-2\sum\limits_{i,j}W_{\mu_1,\underline{x}_j}^3W_{\mu_2,\bar{x}_i}\Bigr]\cr
&-\frac{\mu_2}{4}\int_{\R^3}\Bigl[\bigl(\sum\limits_{j=1}^kW_{\mu_2,\bar{y}_j}+\sum\limits_{j=1}^kW_{\mu_2,\underline{y}_j}\bigr)^4-\sum\limits_{j=1}^kW_{\mu_2,\bar{y}_j}^4
-\sum\limits_{j=1}^kW_{\mu_2,\underline{y}_j}^4
-2\sum\limits_{i\neq j}W_{\mu_2,\bar{y}_j}^3W_{\mu_2,\bar{y}_i}\cr
&-2\sum\limits_{i,j}W_{\mu_2,\bar{y}_j}^3W_{\mu_2,\underline{y}_i}
-2\sum\limits_{i\neq j}W_{\mu_2,\underline{y}_j}^3W_{\mu_2,\underline{y}_i}-2\sum\limits_{i,j}W_{\mu_2,\underline{y}_j}^3W_{\mu_2,\bar{y}_i}\Bigr].
\end{align}
The last three terms in \eqref{eqs3.6} can be estimated exactly as done in the proof of Proposition \ref{proA.4}. Also we can estimate
\begin{align*}
\int_{\Omega_1}W_{\mu_1,\bar{x}_1}^2W_{\mu_2,\bar{y}_1}^2=o(1)\frac{k}{r}e^{-2\sqrt{(1-h^2)((\rho-r\cos\frac{\pi}{k})^2+r^2(\frac{\pi}{k})^2)}}.
\end{align*}

\end{proof}
Acknowledgement: The authors would like to thank Professor Chunhua Wang for the helpful  discussion with her.

\end{document}